\documentclass[11pt]{article}
\usepackage{amssymb}
\usepackage{graphicx}
\usepackage{xcolor} 
\usepackage{tensor}
\usepackage{fullpage} 
\usepackage{amsmath}
\usepackage{amsthm}
\usepackage{verbatim}
\usepackage{enumitem}
\setlist[enumerate]{leftmargin=1.5em}
\setlist[itemize]{leftmargin=1.5em}
\usepackage{hyperref}
\usepackage{seqsplit,mathtools}

\definecolor{green}{rgb}{0,0.8,0} 

\newcommand{\Red}[1]{\begingroup\color{red} #1\endgroup} 

\newtheorem{theorem}{Theorem}[section]
\newtheorem{thm}{Theorem}[section]
\newtheorem{corollary}[theorem]{Corollary}
\newtheorem{lemma}[theorem]{Lemma}
\newtheorem{proposition}[theorem]{Proposition}

\newtheorem{innercustomthm}{Theorem}
\newtheorem*{corollary*}{Corollary}

\newenvironment{customthm}[1]
{\renewcommand\theinnercustomthm{#1}\innercustomthm}
{\endinnercustomthm}

\theoremstyle{definition}

\theoremstyle{remark}
\newtheorem{remark}[theorem]{Remark}

\numberwithin{equation}{section}
\newcommand{\nrm}[1]{\Vert#1\Vert}

\newcommand{\nnrm}[1]{{\vert\kern-0.25ex\vert\kern-0.25ex\vert #1 
		\vert\kern-0.25ex\vert\kern-0.25ex\vert}}

\newcommand{\supp}{{\mathrm{supp}}\,}

\newcommand{\rd}{\partial}
\newcommand{\nb}{\nabla}

\newcommand{\alp}{\alpha}

\newcommand{\dlt}{\delta}

\newcommand{\eps}{\epsilon}
\newcommand{\varep}{\varepsilon}

\newcommand{\lmb}{\lambda}

\newcommand{\tht}{\theta}

\newcommand{\omg}{\omega}

\newcommand{\bbR}{\mathbb R}

\newcommand{\calK}{\mathcal K}

\begin{document}
\bibliographystyle{plain}
 \title{Filamentation near Hill's vortex}
\author{Kyudong Choi\thanks{Department of Mathematical Sciences, Ulsan National Institute of Science and Technology.  
		E-mail: kchoi@unist.ac.kr} 
	\and In-Jee Jeong\thanks{Department of Mathematical Sciences and RIM, Seoul National University.   E-mail: injee\_j@snu.ac.kr}
}

\date\today

\maketitle

\renewcommand{\thefootnote}{\fnsymbol{footnote}}


\renewcommand{\thefootnote}{\arabic{footnote}}

\begin{abstract}  
	For the axi-symmetric incompressible Euler equations, we prove linear in time filamentation near Hill's vortex: there exists an arbitrary small outward perturbation growing linearly for all times. This is based on combining the recent nonlinear orbital stability obtained in \cite{Choi2020} with a dynamical bootstrapping scheme for particle trajectories.  These results rigorously confirm numerical simulations by Pozrikidis \cite{pozrikidis_1986} in 1986. 
\end{abstract}\vspace{1cm} 

\section{Introduction}
\subsection{axi-symmetric Euler equations and Hill's spherical vortex  } 

Filamentation, or formation of long tails, for vortex rings has been frequently observed both in laboratory experiments and numerical simulations (\cite{MM78,pozrikidis_1986,Fukuya,YeChu,Rozi,RoFu,Niel2,GeHe,KX}). In this work, we rigorously confirm such behavior for certain vortex rings employing the three-dimensional  incompressible Euler equations, given by \begin{equation*}\label{euler_velocity}
	\left\{
	\begin{aligned}
		&\partial_t u+( u\cdot \nabla)u+\nabla p =0, \\
		&\nabla \cdot u = 0,\quad x\in\mathbb{R}^3,\quad t \ge 0,
	\end{aligned}
	\right.
\end{equation*} where $u(t,\cdot):\bbR^3\rightarrow\bbR^3$ is the fluid velocity. 
Taking the curl, the above system can be written in terms of the vorticity ($\omg := \nb\times u$) as \begin{equation}\label{euler_vor}
	\partial_t \omega+( u\cdot \nabla)\omega= ( \omega\cdot \nabla) u. \\
\end{equation} 
We assume that the flow is \textit{axi-symmetric without swirl},  which means that 
$$\omg(x)=\omega^\theta(r,z)e_\theta(\theta),\quad x\in\mathbb{R}^3,\quad (r,\theta,z):\mbox{cylindrical coordinate in }\mathbb{R}^3$$(see Subsection \ref{subsec_axi-sym} for the detail). Then, the equations \eqref{euler_vor} simplify to an active scalar equation  \begin{equation}\begin{split}\label{3d_Euler_eq_intro}
		{\partial_t}\xi+ u \cdot \nabla \xi&=0, 
	\end{split}
\end{equation} where  the corresponding \textit{relative} vorticity $\xi$ is defined by 
$$\xi(t,x)=\xi(t,r,z) :=\frac {\omega^\theta(t,r,z)}{r}.$$ 
 In \eqref{3d_Euler_eq_intro},  the velocity field   $u$ is determined from the relative vorticity $\xi$ at each moment of time by the axi-symmetric Biot-Savart law $u=\mathcal{K}[\xi]$, which is discussed in detail later (see \eqref{form_omega}, \eqref{form_psi}, \eqref{form_u} in Subsection \ref{subsec_axi-sym}). In this setting, the Hill's vortex $\xi_H$ (of unit strength on the unit ball) is simply defined by
\begin{equation}\label{defn_hill_intro}
 {\xi_H(x)}:=\mathbf{1}_{B }(  x),
\end{equation}
 where
$B$ is the unit ball in $\mathbb{R}^3$ centered at the origin. 
Then, as it was shown in the classical 1894 paper of M. Hill \cite{Hill}, 
$$\xi(t,x) =\xi_H(x+t u_\infty )=\mathbf{1}_{B }(  x+t u_\infty )$$ is a traveling wave solution of
\eqref{3d_Euler_eq_intro} with $$  u_\infty
=-W_H e_{x_3}, \quad W_H:=(2/15).$$  That is, 
$$\omega(t,x)=r \mathbf{1}_B(x-tW_He_{x_3})e_\theta(\theta)=-x_2 \mathbf{1}_B(x-tW_He_{x_3})e_{x_1}+x_1 \mathbf{1}_B(x-tW_He_{x_3})e_{x_2}$$
is a solution of  \eqref{euler_vor}. 

\subsection{Main results}

Using the cylindrical coordinate system $(r,\theta,z)$, 
we say that a scalar function $f:\mathbb{R}^3\to\mathbb{R}$ is axi-symmetric
if it has the form of $f(x)=f(r,z)$, 
 and  
a subset $A\subset\mathbb{R}^3$ is  axi-symmetric if the characteristic function
${\mathbf{1}}_A:\mathbb{R}^3\to\mathbb{R}$ is axi-symmetric. Let us now state the main result of this paper: 
\begin{customthm}{A}\label{thm:A}
	There exists an arbitrarily $L^1$-small and localized perturbation from  Hill's vortex such that the corresponding solution exhibits linear in time filamentation for all times. 
\end{customthm} The precise (and quantitative) statement 
will become clearer as the text progresses.
 By a $L^1$-small and localized perturbation from Hill's vortex, we mean that $$\nrm{\xi_{0}-\xi_{H}}_{L^1(\mathbb{R}^3)} \le \eps \quad\mbox{and}\quad  \supp(\xi_{0}-\xi_{H}) \subset \left(B(1+\eps)\setminus B(1-\eps)\right)$$ for some small $\eps>0$, respectively. Here, $B(a), a>0$ is the ball of radius $a$ centered at the origin. Moreover, by linear in time {filamentation}, we simply mean that \begin{equation*}
\begin{split}
	\mathrm{diam} (\supp(\xi(t,\cdot))) \gtrsim t. 
\end{split}
\end{equation*} In the statement of Theorem \ref{thm:A}, we can take the initial data to be a smooth patch, i.e. $\xi_0 = {\mathbf{1}}_{A_0}$ with $C^{\infty}$--smooth boundary $\partial A_{0}$. Since $\xi$ is simply being transported by the flow in \eqref{3d_Euler_eq_intro}, the solution takes the form $\xi(t) = {\mathbf{1}}_{A(t)}$ with some open set $A(t)$. We shall refer to such a solution as a vortex patch. It is known that the patch boundary $\partial A(t)$ remains $C^\infty$--smooth for all times (\cite{Hu,Raymond}). As an immediate consequence of Theorem \ref{thm:A}, we have the following:
\begin{corollary}[Infinite time perimeter growth for patches; {{cf. \cite[Figure 2]{pozrikidis_1986}}}]\label{cor:peri_growth}
	{For any $\dlt'>0$, there exists an axi-symmetric open set  $A_{0}\subset\mathbb{R}^3$ with $C^\infty$-smooth boundary satisfying	
	 $$	B(1-\dlt') \subset A_{0} \subset B(1+\dlt') $$}
	 such that the corresponding patch solution $\xi(t)={\mathbf{1}}_{A(t)}$ of \eqref{3d_Euler_eq_intro} for the initial data $\xi_0={\mathbf{1}}_{A_0}$ satisfies 
	 \begin{equation*}
		\begin{split}
			\mathrm{diam}(A(t)),	\quad \mathrm{perim}(A_{cs}(t)) \gtrsim t,\quad \forall t\geq 0,
		\end{split}
	\end{equation*} where	
$A_{cs}(t)$ is the cross-section
 of the axi-symmetric set  $A(t)$  in the half-space 
 $\mathbb{R}^2_+ =\{(r,z)\in\mathbb{R}^2\,|\, r>0,\quad z\in\mathbb{R}\}$.
\end{corollary} 
On the other hand, in Theorem \ref{thm:A}, we can take the initial data $\xi_{0}$ to be a $C^{\infty}$-smooth axi-symmetric function in $\bbR^{3}$. (It implies that the vorticity $\omg = \omg^\tht e_\tht = \xi re_\tht$ is a $C^{\infty}$--smooth function in $\bbR^{3}$.) This gives the infinite time growth for the gradient of $\xi$:
\begin{corollary}[Infinite time gradient growth for smooth solutions]
\label{cor:hessian_growth}
		 For any $\dlt'>0$, there exists a $C^{\infty}$-smooth  and compactly supported axi-symmetric non-negative initial data $\xi_{0}$
with 	
$$
\supp(\xi_{0}-\xi_{H}) \subset \left(B(1+\dlt')\setminus B(1-\dlt')\right)
$$
		  such that the associated solution $\xi(t)$ of \eqref{3d_Euler_eq_intro} and the corresponding vorticity $\omg(t)$ of \eqref{euler_vor} satisfy the growth \begin{equation*}
		\begin{split}
			\nrm{ \rd_{r}\xi(t,\cdot)}_{L^{\infty}} , \quad \nrm{\nb^2\omg(t,\cdot)}_{L^{\infty}}\gtrsim  {\sqrt{t}},\quad \forall t\geq 0.
		\end{split}
	\end{equation*}  
\end{corollary} To the best of our knowledge, infinite time perimeter and derivative growth for the incompressible Euler equations on $\bbR^{3}$ have not been obtained before, with compactly supported and smooth vorticity
(\textit{cf}. see \cite{CJ} for perimeter growth in $\mathbb{R}^2$ up to finite time).
\\ 

A key step in the proof of Theorem \ref{thm:A} is to establish that small perturbations of Hill's vortex travel at about the same speed as Hill's vortex  for \textit{all} times. 

\begin{customthm}{B}\label{thm:B} We have the following statements.\\  \ \\ 
	(I) For any $\varep>0$, there exists $\dlt=\dlt(\varep)>0$ such that the following holds. Given  axi-symmetric $\xi_{0}$ satisfying $\xi_{0}, r\xi_{0} \in L^{\infty}(\bbR^3)$, $\xi_0\geq 0$, 
	 and \begin{align*}
		\|\xi_0-\xi_H\|_{L^1\cap L^2(\mathbb{R}^3)}
		+\|r^2(\xi_0-\xi_H)\|_{L^1(\mathbb{R}^3)} 
		\leq \delta,
	\end{align*} the corresponding solution $\xi(t)$ of \eqref{3d_Euler_eq_intro} 
	satisfies
	\begin{align}\label{orb_est}
		\|\xi(\cdot+\tau(t) e_{z},t)-\xi_H\|_{L^1\cap L^2(\mathbb{R}^3)}
		+\|r^2(\xi(\cdot+\tau(t) e_{z},t)-\xi_H)\|_{L^1(\mathbb{R}^3)} \leq \varepsilon,\quad t\geq 0
	\end{align}  for some shift function $\tau:[0,\infty)\to\mathbb{R}$ satisfying
	$\tau(0)=0$. 	Here, $\|\cdot\|_{L^1\cap L^2}$ means $\|\cdot\|_{L^1}+\|\cdot\|_{L^2}$.
\ \\
	
	\noindent (II) 
	There exists an absolute constant $\varepsilon_0 >0$ such that, for each $M>0$, 
	if   $\|\xi_0\|_{L^\infty}\leq M$ and if $\varepsilon\in(0,\varepsilon_0)$, then the shift function $\tau$ above satisfies   
	 \begin{equation}\label{main_est}
		\begin{split}
			|\tau(t)-W_Ht| \le C_M {\varep^{\frac12}(t +1)},\quad t\geq 0,
		\end{split}
	\end{equation}  where $C_M>0$ is a constant depending only on $M$. 
\end{customthm} 

{The orbital stability statement $(I)$ above was already proved in \cite{Choi2020} based on the classical variational method (in Eulerian description). Our contribution lies on  the key estimate \eqref{main_est} of  $(II)$ based on a dynamical  bootstrap argument on particle trajectories (\textit{i.e.}  Lagrangian approach).\\

As a last application, we prove exponential in time growth of $\nrm{\omg}_{L^\infty}$ for finite (but arbitrarily large) time:
\begin{corollary}[Growth in vorticity for smooth solutions]\label{cor:vor:grow}
 For each $L>1$, there exists a $C^{\infty}$--smooth  and compactly supported initial data $\omega_{0}=\omega_{0}^\theta e_\theta$ satisfying \begin{equation}\label{cor_vor_grow_concl}
		\begin{split}
	\omega_{0}^\theta\geq 0,\quad	 \nrm{\omg_0}_{L^{\infty}(\mathbb{R}^3)}\leq 1, \quad\mbox{and}\quad   {\sup}_{t\in[0,c\log(L)]}\nrm{ \omg(t,\cdot)}_{L^{\infty}(\mathbb{R}^3)}\geq L,
		\end{split}
	\end{equation}    where $\omega(t)$ is the corresponding smooth solution of \eqref{euler_vor}  and $c>0$ is an absolute constant. 
\end{corollary} 
} 

\subsection{Related works}\label{subsec:related}

\medskip

\noindent \textbf{Experiments and numerical simulations}. Axi-symmetric regions of fluid with {highly concentrated} vorticity can be observed in a variety of settings including jets and atmospheric plumes. Especially at high Reynolds number, such vortex structures are effectively modelled by Hill's vortex (\cite{WePa,harper_moore_1968,Turner}). Linearized dynamics near Hill's vortex was studied by Moffatt-Moore in \cite{MM78}, and Pozrikidis carried out detailed numerical simulations for axi-symmetric perturbations from Hill's vortex (\cite{pozrikidis_1986}).  The simulations show that an initially prolate perturbation develops into a long tail (see \cite[Figure 2]{pozrikidis_1986}), whereas an oblate perturbation catches up to the front boundary of the vortex (see \cite[Figure 7]{pozrikidis_1986}). {This behavior is rigorously confirmed in the current work; see the statements of Corollary \ref{cor:peri_growth} and \ref{cor:ins_growth} which correspond to Figure 2 and 7 in \cite{pozrikidis_1986}, respectively.} Further numerical computations demonstrating filamentation near the Hill's vortex can be found in \cite{Fukuya,YeChu,Rozi,RoFu}.

\medskip

\noindent  \textbf{Nonlinear orbital stability}. 
	The nonlinear \textit{orbital} stability (stability up to a \textit{translation}) obtained in \cite{Choi2020} is based on the variational principles suggested by Benjamin \cite[Section I]{Ben76}  in 1976. The variational setting became more  concrete     in Friedman-Turkington \cite{FT81}, who showed existence of a maximizer which is  a vortex ring. Independently, Amick-Fraenkel \cite{AF86} proved uniqueness of Hill's vortex among vortex rings.  Wan \cite{Wan88} in 1988 showed that Hill’s vortex  is a nondegenerate local maximum of the kinetic energy under certain constraints. Very recently, \cite{Choi2020} proved that Hill's vortex is Lyapunov orbitally stable in $(L^1+L^2+\mbox{impulse})$-norm. The key idea is to connect  the above classical results (existence \cite{FT81}, uniqueness \cite{AF86}) as well as to employ concentrated compactness argument due to \cite{Lions84a} and a recent existence result \cite{NoSe} of (renormalized sense) weak solutions  (see also \cite{BNL13},  \cite{AC2019}).\\
	
	In solitary waves,  
	  an orbital stability   appeared first in \cite{Ben72}, \cite{Bona75} for the   generalized  KdV equation (see the survey paper \cite{Tao09} for related references).    
   Stability up to a translation  can be found even  for inviscid/viscous shocks in   conservation laws. For instance,  we refer to the classical paper \cite {MR0101396} for asymptotic stability up to a translation  in $L^1$-setting (also see \cite{MR2217605} and references therein).      In $L^2$-setting, we mention     \cite{LegerV},   \cite{Kang-V-NS20}  for   systems  such as
  compressible Euler/Navier-Stokes systems  
 and \cite{ckkv2019} for certain Keller-Segel type systems.   
 
\medskip

\noindent \textbf{Vortex patch dynamics: growth of the support}.  { 
When considering \textit{non-negative} compactly supported relative vorticity in 3d axi-symmetric setting \eqref{3d_Euler_eq_intro}, 
 vorticities must be  confined in the region 
$\{r\leq C t^{1/4}\log t\}$, which was proved in \cite{Maffei2001}. Later,  \cite{Ren2008} showed that  vorticities cannot move in the opposite direction ($-z$ direction) faster than the rate $(t\log t)^{1/2}$. In sum, such vorticities  are confined in the region
$$
\{
r\leq C t^{1/4}\log t,\quad -C(t\log t)^{1/2}\leq z\leq Ct
\}
$$ The ideas of   confinement  for such \textit{one-signed} vorticities were going back to  \cite{M94}, \cite{ISG99}, \cite{Serfati_pre} in the planar case $\mathbb{R}^2$. The sign-condition makes 
the conservation of impulse $\int_{\mathbb{R}^2}|x|^2\omega dx$ to   control confinement of vorticities. It is notable that there is an example of linear growth when dropping the condition (see  \cite[Sec. 3]{ISG99}).
For confinement on other   two-dimensional  domains, we refer to \cite{ILL2003}, \cite{ILL2007}, \cite{CD2019}.}
 
\medskip

\noindent \textbf{Growth of vorticity in three-dimensional Euler equations}. Compared with the two-dimensional case, there are fewer rigorous results pertaining to the growth of the vorticity in some norm for the three-dimensional Euler equations. Infinite time linear growth of $\nrm{\nb^2\omg}^2_{L^\infty}$ and arbitrarily long but finite time exponential growth of $\nrm{\omg}_{L^\infty}$ are achieved in the current work for $C^\infty_c(\bbR^3)$--vorticity. In the presence of a physical boundary, obtaining growth is strictly simpler since the vorticity may not vanish on the boundary, leading to a stable growth mechanism; see \cite{EJE,EJO,Do,ChenHou} and the references therein. When the boundary is absent but if the domain is periodic with respect to an axis, then it is well-known that one can use the  $2\frac{1}{2}$--dimensional flow construction to obtain linear in time growth of the 3D vorticity (see \cite{EM1,MaB,BT,JY}). In the case of $\bbR^{3}$, a recent groundbreaking work of Elgindi \cite{Elgindi19} achieved the finite-time singularity formation for $C^\alp$--vorticity, in the class of axi-symmetric no-swirl flows. In the work \cite{Elgindi19}, it is essential that the vorticity is of limited smoothness. Still, it is expected that the ``hyperbolic flow scenario'' utilized in \cite{Elgindi19} is responsible for an infinite time growth of $\nrm{\omg}_{L^{\infty}}$; see \cite{Child,ChildGil}. 

\subsection{Ideas of the proof}\label{subsec:ideas}

Let us explain the main ingredients of the proof.  
\begin{itemize}
	
	\item \textbf{Traveling speed of the perturbation (Theorem \ref{thm:B} $(II)$).}
	The orbital stability from \cite{Choi2020} (or just Theorem \ref{thm:B} $(I)$) does not say where   the orbit element at each time is, since the position $\tau(t)$ in \eqref{orb_est} is \textit{implicitly} given from a contradiction argument.\\

In this paper, we overcome this weak point by applying a bootstrap argument to obtain \eqref{main_est}. We take two bootstrap hypotheses for each time $t_0$:\\
One is for the difference of the shift function in time $$|\tau(t) - (\tau(t_0) + {W_H}(t-t_0))|$$ while the other is for the particle trajectory $\phi(t,(t_0,{B}^{\tau(t_0)}))$ initiated from the ball $$B^{\tau(t_0)}=\{
|x-\tau(t_0)e_{x_3}|<1
\}.$$ In Proposition \ref{prop:travel-speed_hill}, we prove that they are well-controlled for short time by assuming smallness at time $t_0$. To finish the proof, a technicality appears since a shift function satisfying \eqref{orb_est} is not uniquely determined. For instance, arbitrary but small perturbations of a given shift satisfying \eqref{orb_est} can also satisfy the required stability \eqref{orb_est}. Such perturbed shift functions need not be   continuous   at all.
 This difficulty is removed by noting that   discontinuity (if exists) is limited to at most small jumps (Lemma \ref{lem:no_high_jump}).

	\item \textbf{Dynamics of perturbations}.
	Once we show that perturbation behaves like   Hill's vortex (due to \eqref{main_est}), we freely use the explicit information on the velocity of   Hill's vortex \eqref{vel_explicit}. In particular, a fluid particle in a perturbation which was initially behind   Hill's vortex moves further and further away from the vortex core. It produces a long tail (filamentation), which gives Theorem \ref{thm:A} with the following consequences:
\begin{itemize}

			\item {perimeter growth in infinite time (Corollary \ref{cor:peri_growth})},
			\item {growth in Hessian of smooth vorticity in infinite time (Corollary \ref{cor:hessian_growth})}.
\end{itemize}
On the other hand, a particle in a perturbation which was very close to the front of the core   gives 
\begin{itemize}
	\item {growth in  vorticity up to finite time (Corollary \ref{cor:vor:grow})}. \\
\end{itemize} In the proof of this corollary, the relative vorticity $\xi(t)$ producing the growth \eqref{cor_vor_grow_concl} is chosen to satisfy 
$$\|\xi(t)\|_{L^\infty}=\|\xi_0\|_{L^\infty}\sim L$$ while
$$		 		\|\xi_0-\xi_H\|_{L^1\cap L^2(\mathbb{R}^3)}		 
+\|r^2(\xi_0-\xi_H)\|_{L^1(\mathbb{R}^3)} 
\quad\mbox{is sufficiently small.}
$$ Roughly speaking, we set a peak of order $L$ on the initial data $\xi_0$ near 
the front point i.e. $r=0, z=1$. Then, this peak climbs up the front surface of Hill's vortex at least for a finite time. Recalling $\omega=\omega^\theta e_\theta= r \xi e_\theta$, we obtain growth in the vorticity vector $\omega$; see Figure \ref{fig1}.
\end{itemize}

\begin{remark}

	We remark that the main estimate \eqref{main_est} can  be   proved  in a simpler argument if we are allowed to assume further that the initial data $\xi_0$ is supported in $\{|z|\leq Z_0\}$ for some constant $Z_0>0$. The corresponding estimate, of course, depends on the parameter $Z_0$. The key idea would be to use the identity on the the speed of the center of mass in $z$-direction:
	$$
	\frac{d}{dt}\int_{\mathbb{R}^3} z\xi (t,x)\,dx=\int_{\mathbb{R}^3} u^z(t,x)\xi (t,x)\,dx.
	$$  
\end{remark}

\subsection{Organization of the paper} In the rest of paper, we introduce preliminaries in Section 2 including orbital stability due to \cite{Choi2020}. Then, in Section 3, we prove the main estimate \eqref{main_est}. As a consequence, Corollary \ref{cor:vor:grow} follows. In Section 4, we present filamentation results (Corollaries \ref{cor:peri_growth}, \ref{cor:hessian_growth}).  Unless otherwise specified, the $L^p$ norms are always taken with respect to the Lebesgue measure on $\bbR^3$.

\section{Preliminaries}

\subsection{Axi-symmetric Biot-Savart law} \label{subsec_axi-sym} 
A  vector field $ u
$ is  called \textit{axi-symmetric without swirl}
if 
it has the form of   $$ u(x)=u^r(r,z)e_r(\theta)  +u^z(r,z)e_z,
$$ for 
\begin{equation*}
e_r(\theta)=(\cos\theta,\sin\theta,0),
\quad e_\theta(\theta)=(-\sin \theta, \cos \theta,0),
\quad e_z=(0,0,1),
\end{equation*}
where $(r,\theta,z)$ is the cylindrical coordinate to the Cartesian coordinate $x=(x_1,x_2,x_3)$, i.e. $x_1=r\cos\theta,\, x_2=r\sin\theta,\,x_3=z.$\\

For given axi-symmetric \textit{nice} function $\xi=\xi(r,z)$, we set 
the vorticity vector field 
\begin{equation}\label{form_omega}
 {\omega} : =  \omega^\theta e_\theta(\theta),
\quad \omega^\theta:=r\xi,
\end{equation} and  set
\begin{equation}\label{form_psi}
 \phi:
=\frac{1}{4\pi|x|}*_{\mathbb{R}^3}\omega.
 \end{equation}
 Then $\phi$ has only its angular component $\phi^\theta e_\theta$,
 and the velocity 
\begin{equation}\label{form_u}
 u=\mathcal{K}[\xi] :=\nabla \times\phi
\end{equation} 
 is axi-symmetric without swirl and divergence-free. In particular, it satisfies  ${\omega} =\nabla\times  u$, i.e. $$
(\partial_z u^r-\partial_r u^z) =  \omega^\theta . 
$$ 
For an explicit representation of the axi-symmetric Biot-Savart law $u=\mathcal{K}[\xi]$, we refer to 
 \cite[Section 1]{FT81}, \cite{Sverak_lecture}, \cite{Do}.\\

We frequently use the following  estimate on velocity in terms of relative vorticity which was   proved very recently:
\begin{lemma}[{{\cite[Feng--Sverak (1.23)]{FengSverak}}}]  \label{lem:FS} For $\xi \in L^1 \cap L^\infty$ with $r^2\xi \in L^1$, $\calK[\xi]$ satisfies 
	\begin{equation}\label{est_vel}
		\|\mathcal{K}[\xi]\|_{L^\infty(\mathbb{R}^3)}\leq C_{0}\|r^2\xi\|_{L^1(\mathbb{R}^3)}^{1/4}
		\|\xi\|_{L^1(\mathbb{R}^3)}^{1/4}\|\xi\|_{L^\infty(\mathbb{R}^3)}^{1/2},
	\end{equation} where $C_0>0$ is an absolute constant.
\end{lemma}

\subsection{Revisited Hill's spherical vortex}\label{subsec_def_hill}
The  axi-symmetric velocity $u=(u_H^re_r+  u_H^ze_z)$ of Hill's vortex $\xi_H=\mathbf{1}_{\{|x|\leq1\}}$ \eqref{defn_hill_intro} is explicitly written by 
 \begin{equation}\label{vel_explicit}\begin{split}
		&u_H^r=-\frac{\partial_z\psi_H}{r}=
		\begin{cases}
			&\frac{3}{2}Wrz,\quad |x|\leq 1,\\
			&\frac{3}{2}W\frac{rz}{|x|^5},\quad |x|>1, \end{cases}\\
		&u_H^z=\frac{\partial_r\psi_H}{r}=
		\begin{cases}
			&\frac{W}{2}\left(5-3|x|^2-3r^2\right),\quad |x|\leq 1,\\
			&\frac{W}{|x|^3}\left(1-\frac{3r^2}{2|x|^2}\right),\quad |x|>1, \end{cases}
	\end{split}
\end{equation}  where
the traveling speed $W=W_H$ is equal to ${2}/{15}$. 
\subsection{Lyapunov orbital stability }

We borrow the following orbital stability theorem from \cite{Choi2020}:
\begin{thm}[{{\cite[Theorem 1.2]{Choi2020}}}]\label{thm_hill_gen} 
	For  $\varepsilon>0$, 
	there exists $\delta=\delta(\varepsilon)>0$ such that 
	for 
	any non-negative axi-symmetric  function
	$\xi_0$ satisfying
	\begin{equation}\label{assump_uniq}
		\xi_0, \, r\xi_0\in L^\infty(\mathbb{R}^3)
	\end{equation} and
	\begin{align*}
		\|\xi_0-\xi_H\|_{L^1\cap L^2(\mathbb{R}^3)}
		+\|r^2(\xi_0-\xi_H)\|_{L^1(\mathbb{R}^3)} 
		\leq \delta,
	\end{align*} the corresponding
	solution $\xi(t)$ of \eqref{3d_Euler_eq_intro} 
	for the initial data $ \xi_0$ 
	satisfies
	\begin{align*}\label{conclu_orb_patch_gen}
		\inf_{\tau\in\mathbb{R}}\left\{ 
		\|\xi(\cdot+\tau e_{z},t)-\xi_H\|_{L^1\cap L^2(\mathbb{R}^3)}
		+\|r^2(\xi(\cdot+\tau e_{z},t)-\xi_H)\|_{L^1(\mathbb{R}^3)}  
		\right\}
		\leq \varepsilon \quad \mbox{for all}\quad t\geq0.
	\end{align*} 
\end{thm} 

\begin{remark} \label{rem_uniq_sol} 
	For any
	axi-symmetric initial data 
	\begin{equation*}
		0\leq \xi_0\in 
		(L^1\cap L^2)(\mathbb{R}^3)\quad\mbox{with}\quad
		r^2\xi_0\in   L^1(\mathbb{R}^3),
	\end{equation*}
	 the corresponding global-in-time weak solution exists uniquely if we assume    the  extra condition \eqref{assump_uniq} on the relative vorticity $\xi_0$ (by  \cite{UI}, or  see \cite{Raymond}, \cite{Danchin}).  
In particular, the   solution  preserves in time  the 
quantities
 below (\textit{e.g.} see \cite{NoSe} or  see \cite[Lemma 3.4]{Choi2020}):
\begin{enumerate}
	\item     Impulse
	$ \quad \frac{1}{2}\int_{\mathbb{R}^3}r^2\xi  dx.
	$
	\item Kinetic energy $\quad 		\frac{1}{2}\int_{\mathbb{R}^3} |u|^2 dx. $
	\item  	
	$L^p-$norm   $\quad \|\xi\|_{L^p(\mathbb{R}^3)},\quad p\in[1,\infty].$\end{enumerate}
\end{remark}

\subsection{Flow map}

\begin{proposition}
	For axi-symmetric initial data satisfying $\xi_0 \in L^1\cap L^\infty$ and $\omg^\tht_0 = r\xi_0 \in L^\infty$, there exists a unique flow $\phi(t,x): [0,\infty) \times \bbR^3\rightarrow\bbR^3 $ satisfying $$\frac{d}{dt}\phi(t,x)=u(t,\phi(t,x)),\quad \phi(0,x)=x $$  and hence 
	$$\xi(t,\phi(t,x))=\xi_0(x).$$ For each $t>0$, $\phi(t,\cdot)$ is a H\"older continuous homeomorphism of $\bbR^3$.  Moreover, the symmetry axis $\{r=0\}$ is invariant by the flow.
\end{proposition}
\begin{proof}
 We recall the following estimate from \cite[Lemma 2]{Danchin} \begin{equation*}
		\begin{split}
			\frac{|u^r(t)|}{r} \le C\nrm{\xi(t)}_{L^{3,1}} \le C\nrm{\xi(t)}_{L^1\cap L^\infty} \le C\nrm{\xi_0}_{L^1\cap L^\infty} .
		\end{split}
	\end{equation*}  Then, from the vorticity equation, we have  \begin{equation*}
	\begin{split}
		\frac{d}{dt} \nrm{\omg^\tht}_{L^\infty \cap L^1} \le C\nrm{\xi_0}_{L^1\cap L^\infty}\nrm{\omg^\tht}_{L^\infty  \cap L^1}, 
	\end{split}
\end{equation*} which gives \begin{equation*}
\begin{split}
	\nrm{\omg^\tht(t)}_{L^\infty \cap L^1} \le \nrm{\omg^\tht_0}_{L^\infty \cap L^1} \exp( Ct \nrm{\xi_0}_{L^1\cap L^\infty} ).
\end{split}
\end{equation*} We have used that the velocity is divergence free. Using the above, we obtain  \begin{equation*}
\begin{split}
	\nrm{u(t)}_{logLip} := \sup_{x \ne x'} \frac{|u(t,x)-u(t,x')|}{|x-x'| \ln(10 + |x-x'|^{-1})}  \le C\nrm{\omg^\tht(t)}_{L^\infty \cap L^1} \le \nrm{\omg^\tht_0}_{L^\infty \cap L^1} \exp( Ct \nrm{\xi_0}_{L^1\cap L^\infty} ). 
\end{split}
\end{equation*} This follows from combining the standard estimate for singular integral operators (\cite{Stein70,MaB}) which gives $\nabla u (t) \in BMO(\bbR^3)$ with the observation that an anti-derivative of a $BMO$ function is log-Lipschitz (see for instance \cite{Bed}). Thanks to the log-Lipschitz estimate, there exists a unique and global-in-time solution to the ODE $$\frac{d}{dt}\phi(t,x)=u(t,\phi(t,x)),\quad \phi(0,x)=x $$ for each $x\in \bbR^3$ (\textit{e.g.} see \cite[Lemma 3.2]{MaPu}). Then one can proceed to show that $\phi(t,\cdot)$ for each $t\ge 0$ is a bijection which is H\"older continuous with H\"older continuous inverse. See for instance \cite[Section 8.2]{MaB}. Lastly, thanks to the axis-symmetry with divergence-free condition, it is easy to see 
$$
u^r|_{r=0}\equiv 0,
$$  which gives the last assertion.
\end{proof} 

From now on, for given time-dependent velocity field $u(\cdot_t,\cdot_x)$, we denote $\phi$ by  the particle trajectory map $\phi(t,(t_0,x))$  obtained from solving the following ODE system with any initial time $t_0\geq 0$: 
	\begin{equation}\label{traj}
		\frac{d}{dt}\phi(t,(t_0,x))=u(t,\phi(t,(t_0,x)))\quad\mbox{for }t>0\quad\mbox{and } \quad \phi(t_0,(t_0,x))=x\in\mathbb{R}^3.
	\end{equation}
When $t_0=0$ is considered, we simply call $\phi(t,x)=\phi(t,(0,x))$.

\section{Traveling speed of the perturbation}

The primary goal of this section is to obtain estimates on the shift function $\tau(t)$, see \eqref{eq:no_high_jump} and \eqref{eq:travel-est_hill}. With these key estimates proved in Subsection 3.1, we conclude Theorem \ref{thm:B} and Corollary \ref{cor:vor:grow} in Subsections 3.2 and 3.3, respectively. 

 \subsection{Estimates on shift $\tau(t)$}
For any $\tau\in\mathbb{R}$, we define
\begin{equation}\label{def_f}
f(\tau):=
|B\triangle B^\tau|,
\end{equation}
where we 
denote
 $$ {{B}^{\tau}}:=\{x\in\mathbb{R}^3\,:\, |x-\tau e_{x_3}|<1\}, \quad B^{0}=B.$$
We can explicitly compute $f$:
$$f(\tau)=
4\pi \int_0^{|\tau|/2}(1-s^2)  ds
=\frac{\pi}{6}|\tau|(12-\tau^2),\quad |\tau|\leq 2,$$
and $f(\tau)=2\cdot \frac 4 3 \pi$ for any $|\tau|\geq 2$.
We note that $f$ is strictly increasing on $[0,2]$ and
\begin{equation}\label{claim_hill_1}
  \frac 4 3 \pi |\tau|\leq f(\tau),
 \quad\mbox{for any}\quad  |\tau|\leq 2.
\end{equation}

\begin{lemma}\label{lem:no_high_jump} 
There exist absolute constants $\varepsilon_1>0$, $ \tilde{K}>0$ such that 
for each $M>0$, there is a constant  $ \tilde{C}=\tilde{C}(M)>0$ satisfying the following statement: \\

For any non-negative axi-symmetric $\xi_{0}$ satisfying $\xi_{0}, r\xi_{0} \in L^{\infty}(\bbR^3)$, and
$\|\xi_0\|_{L^\infty}\leq M$,
if  there is a shift function $\tau(\cdot):[0,\infty)\to\mathbb{R}$ satisfying
\begin{align*}
		\sup_{t\geq 0}\left\{\|\xi(\cdot+\tau(t) e_{z},t)-\xi_H\|_{L^1\cap L^2(\mathbb{R}^3)}
		+\|r^2(\xi(\cdot+\tau(t) e_{z},t)-\xi_H)\|_{L^1(\mathbb{R}^3)}\right\} \leq \varepsilon 
	\end{align*}  for some $\varepsilon\in(0,\varepsilon_1)$,
	then 
the function $\tau$  satisfies
 \begin{equation}\label{eq:no_high_jump}
	\begin{split}
		|\tau(t) - \tau(t')  | \leq \tilde{K}\cdot \varepsilon^{1/2},\quad t,t'\geq 0
	\end{split}
\end{equation} whenever $|t-t'|\leq \tilde{C}\cdot \varepsilon^{1/2}$.
\end{lemma}
\begin{proof}
\begin{enumerate}

\item We denote
$$\xi_H^{\tau}:=\xi_H(\cdot-\tau e_{{z}})={\mathbf{1}}_{B^{\tau}}(\cdot),\quad \tau\in\mathbb{R}.$$
Let $\varepsilon_1\in(0,1/4)$. We will take $\varepsilon_1$ (again) small enough in the proof.
We take $\xi$ to be a solution satisfying all the assumptions of this lemma  with some $\varepsilon\in(0,\varepsilon_1)$ and $M>0$,
and define sets
$$G_\varepsilon(t):=\{
x\in \mathbb{R}^3\setminus B^{\tau(t)}\,:\,|\xi(t,x) -1|\leq \varepsilon^{1/2}
\},$$
$$E_\varepsilon(t):=\{
x\in B^{\tau(t)}\,:\,|\xi(t,x) -1|\leq \varepsilon^{1/2}
\}\subset B^{\tau(t)}\quad\mbox{and}\quad F_\varepsilon(t):= B^{\tau(t)}\setminus E_\varepsilon(t),\quad t\geq 0.$$
Then, by Chebyshev's inequality,  we have
$$
|G_\varepsilon(t)|\leq|\{
x\in \mathbb{R}^3\setminus B^{\tau(t)}\,|\,\xi(t,x)\geq 1/2
\}|\leq 2\|\xi(t)\|_{L^1(\mathbb{R}^3\setminus B^{\tau(t)})}
\leq 2\|\xi(t)-\xi_H^{\tau(t)}\|_{L^1 }\leq 2\varepsilon
$$ and
$$|F_\varepsilon(t)|\leq \|\xi(t)-\xi_H^{\tau(t)}\|^2_{L^2}\cdot\varepsilon^{-1}\leq \varepsilon,\quad t\geq 0.$$
Then  we can compute for any $t,t'\geq 0$,
  \begin{equation}
\begin{split}\label{est_b_tau}
  &|B^{\tau(t')}\triangle B^{\tau(t)}|= \|\xi_H^{\tau(t')}-\xi_H^{\tau(t)}\|_1
 =\|\xi_H^{\tau(t')}{\mathbf{1}}_{{B}^{\tau(t')}}-\xi_H^{\tau(t)}{\mathbf{1}}_{{B}^{\tau(t)}}\|_{L^1}\\
 &\quad  \leq \|\xi_H^{\tau(t')}{\mathbf{1}}_{{B}^{\tau(t')}}-\xi(t'){\mathbf{1}}_{{B}^{\tau(t')}}\|_{L^1}+\|\xi(t'){\mathbf{1}}_{{B}^{\tau(t')}}-\xi(t){\mathbf{1}}_{{B}^{\tau(t)}}\|_{L^1}+\|\xi(t){\mathbf{1}}_{{B}^{\tau(t)}}-\xi_H^{\tau(t)}{\mathbf{1}}_{{B}^{\tau(t)}}\|_{L^1}\\
&\quad = \|\xi_H^{\tau(t')}-\xi(t')\|_{L^1({B}^{\tau(t')})}+ \|\xi(t'){\mathbf{1}}_{{B}^{\tau(t')}}-\xi(t){\mathbf{1}}_{{B}^{\tau(t)}}\|_{L^1}+  \|\xi(t)-\xi_H^{\tau(t)}\|_{L^1({B}^{\tau(t)})}\\
&\quad\leq  \varepsilon + \|\xi(t'){\mathbf{1}}_{{B}^{\tau(t')}}-\xi(t){\mathbf{1}}_{{B}^{\tau(t)}}\|_{L^1}+\varepsilon\\&\quad\leq 2\varepsilon + \|\xi(t'){\mathbf{1}}_{E_\varepsilon(t')}-\xi(t){\mathbf{1}}_{E_\varepsilon(t)}\|_{L^1}+\|\xi(t'){\mathbf{1}}_{F_\varepsilon(t')}\|_{L^1}+\|\xi(t){\mathbf{1}}_{F_\varepsilon(t)}\|_{L^1}\\
&\quad\leq 2\varepsilon   +\|\xi(t'){\mathbf{1}}_{E_\varepsilon(t')}- {\mathbf{1}}_{E_\varepsilon(t')}\|_{L^1} +\underbrace{|
 {E_\varepsilon(t')}\triangle {E_\varepsilon(t)}|}_{=:I(t',t)}+\| {\mathbf{1}}_{E_\varepsilon(t)}-\xi(t){\mathbf{1}}_{E_\varepsilon(t)}\|_{L^1}\\ &\quad\quad+ 
\|\xi(t')\|_{L^2}|{F_\varepsilon(t')}|^{1/2}+
\|\xi(t)\|_{L^2}|{F_\varepsilon(t)}|^{1/2}\\
&\quad\leq 2\varepsilon +\|\xi(t') - \xi_H^{\tau(t')} \|_{L^1({E_\varepsilon(t')})}  + I(t',t)+\|\xi(t) - \xi_H^{\tau(t)} \|_{L^1({E_\varepsilon(t)})}  + 
2(\|\xi_H\|_{L^2}+\varepsilon)\varepsilon^{1/2}\\
&\quad\leq C_0\varepsilon^{1/2}   + I(t',t),   
\end{split}\end{equation} where $C_0>0$ is an absolute constant. Let us now proceed to prove $I(t',t) \lesssim \varep^{\frac12}$. To this end, we begin with 
\begin{align*}   {E_\varepsilon(t')}&=\left( \phi(t',(t, {E_\varepsilon(t)})\cup  \phi(t',(t, {G_\varepsilon(t)})\right)\cap B^{\tau(t')}\\   &
\subset \phi(t',(t, {E_\varepsilon(t)})\cup  \phi(t',(t, {G_\varepsilon(t)}),
\end{align*} 
 where 
	$\phi$ is the particle trajectory map \eqref{traj}.
Thus we compute
  \begin{align*}
 | {E_\varepsilon(t')}\setminus {E_\varepsilon(t)}|& \leq 
  |\phi(t',(t, {E_\varepsilon(t)})\setminus {E_\varepsilon(t)}|
  +  |\phi(t',(t, {G_\varepsilon(t)})|\\
  &\leq 
  |\phi(t',(t, B^{\tau(t)}))\setminus {E_\varepsilon(t)}|
  +  |{G_\varepsilon(t)}|\\
   &\leq 
  |\phi(t',(t, B^{\tau(t)}))\setminus {E_\varepsilon(t)}|
  +2\varepsilon.
\end{align*} 
Due to the estimate \eqref{est_vel}, 
the flow speed is uniformly bounded
\begin{equation}\begin{split}\label{unif_speed_h}
		\|u(t)\|_{L^\infty(\mathbb{R}^3)}&\leq C_{0}\|r^2\xi(t)\|_{L^1(\mathbb{R}^3)}^{1/4}
		\|\xi(t)\|_{L^1(\mathbb{R}^3)}^{1/4}\|\xi(t)\|_{L^\infty(\mathbb{R}^3)}^{1/2} \\
&=  C_{0}\|r^2\xi_0\|_{L^1(\mathbb{R}^3)}^{1/4}
		\|\xi_0\|_{L^1(\mathbb{R}^3)}^{1/4}\|\xi_0\|_{L^\infty(\mathbb{R}^3)}^{1/2} \\
		&\leq C_{0}(\varepsilon+\|r^2\xi_H\|_{L^1(\mathbb{R}^3)})^{1/4}
		(\varepsilon+\|\xi_H\|_{L^1(\mathbb{R}^3)})^{1/4}M^{1/2}\\
		&\leq C(\varepsilon_1+1)^{1/4}
		(\varepsilon_1+1)^{1/4}M^{1/2}\leq C_1M^{1/2},
\end{split}  \end{equation} where $C_1>0$ is an absolute constant. By denoting
$$
 {{B}^{\tau,a}}:=\{x\in\mathbb{R}^3\,|\, |x-\tau e_{x_3}|<a\},\quad a>0,$$
 we get
\begin{align*}
    \phi(t',(t, B^{\tau(t)}))\subset B^{\tau(t),\,{1+C_1M^{1/2}|t-t'|}},
\end{align*} which gives
  \begin{align*}
 | {E_\varepsilon(t')}\setminus {E_\varepsilon(t)}|&  \leq 
  |B^{\tau(t),\,{1+C_1M^{1/2}|t-t'|}}\setminus {E_\varepsilon(t)}|
  +  2\varepsilon\\
  & \leq 
  |B^{\tau(t),\,{1+C_1M^{1/2}|t-t'|}}\setminus B^{\tau(t)}|+|B^{\tau(t)}\setminus {E_\varepsilon(t)}|
  +  2\varepsilon\\
  & \leq  \frac 4 3 \pi \left( \left(1+C_1M^{1/2}|t-t'|\right)^3-1^3\right)+\underbrace{| {F_\varepsilon(t)}|}_{\leq \varepsilon}
  +  2\varepsilon\\
   & \leq  C_2  \varepsilon^{1/2} 
\end{align*} whenever $t,t'\geq0$ satisfy
\begin{equation}\label{cond_ep1}
 \left( C_1M^{1/2}|t-t'|\right) \leq \varepsilon^{1/2},
\end{equation}
where $C_2>0$ is an absolute constant. By exchanging the role of $t$ with that of $t'$, we get, for any $t,t'\geq 0$ assuming \eqref{cond_ep1},
  \begin{align*}
 I(t',t)=|
 {E_\varepsilon(t')}\setminus {E_\varepsilon(t)}|+|
 {E_\varepsilon(t)}\setminus {E_\varepsilon(t')}|\leq  2  C_2  \varepsilon^{1/2}.
\end{align*} 
By recalling the definition of $f$ in \eqref{def_f} and 
the estimate \eqref{est_b_tau}, we have shown, for any $t,t'\geq 0$ with \eqref{cond_ep1},
  \begin{align*}
 f(\tau(t')-\tau(t))=|B\triangle B^{\tau(t')-\tau(t)}| =|B^{\tau(t')}\triangle B^{\tau(t)}|& \leq (C_0+2C_2)  \varepsilon^{1/2}.
\end{align*} 
Lastly, we assume $\varepsilon_1>0$ smaller than before in order to have
  \begin{align*}
   (C_0+2C_2)\varepsilon_1^{1/2}<2\cdot\frac 4 3 \pi=f(2).
\end{align*} 
This assumption guarantees 
  \begin{align*}
 |\tau(t')-\tau(t)|\leq 2,
\end{align*}  which, in turn (due to \eqref{claim_hill_1}), gives
  \begin{align*}
\frac 4 3 \pi|\tau(t')-\tau(t)| \leq f(\tau(t')-\tau(t))\leq (C_0+2C_2) \varepsilon^{1/2}.
\end{align*}  \end{enumerate} This gives the statement. 
 \end{proof}

 \begin{proposition}[Traveling speed]\label{prop:travel-speed_hill} 
There exist absolute  constants $\varepsilon_0>0,\, K>0$ such that
for each $M>0$, there is a   constant $\alpha_0=\alpha_0(M)>0$  satisfying the following statement:\\ 

For any non-negative axi-symmetric $\xi_{0}$ satisfying $\xi_{0}, r\xi_{0} \in L^{\infty}(\bbR^3)$, and
$\|\xi_0\|_{L^\infty}\leq M$,
if  there is a shift function $\tau(\cdot):[0,\infty)\to\mathbb{R}$ satisfying
\begin{align}\label{assum_lem_stab} 
	\sup_{t\geq 0}\left\{	\|\xi(\cdot+\tau(t) e_{z},t)-\xi_H\|_{L^1\cap L^2(\mathbb{R}^3)}
		+\|r^2(\xi(\cdot+\tau(t) e_{z},t)-\xi_H)\|_{L^1(\mathbb{R}^3)} \right\}\leq \varepsilon
	\end{align}  for some $\varepsilon\in(0,\varepsilon_0)$,
	then 
the function $\tau$  satisfies
   \begin{equation}\label{eq:travel-est_hill} 
		\begin{split}
			|\tau(t + \alp  ) - \tau(t) - {W} \alp   | \le {K}\cdot \varep^{1/2}
		\end{split}
	\end{equation}   for any $t\geq0$ and for any $\alp \in [0,\alp_0]$.
\end{proposition}

\begin{proof} 

\begin{enumerate}
\item  As before, we denote balls, for $\tau\in\mathbb{R}$ and $a>0$,
$$ {{B}^{\tau,a}}:=\{x\in\mathbb{R}^3\,|\, |x-\tau e_{{z}}|<a\},\quad 
 {{B}^{\tau}}:= {{B}^{\tau,1}}.$$ We will use the following simple estimate later:
\begin{equation}\label{kappa_est}
|B^{0,1+\kappa}\setminus B^{0,1}|=\frac 4 3 \pi \left( (1+\kappa)^3-1^3\right)\leq 30\kappa\quad \mbox{for any }\quad 0\leq \kappa<1/10.
\end{equation} 
 We recall  that $\xi_H(t,x):=\xi_H(x-t{W}e_{{z}})$ (with some abuse of notation) is a traveling wave solution of \eqref{3d_Euler_eq_intro}. The velocity of  Hill's vortex 
 $$u_H(t, x):=\mathcal{K}[\xi_{H}(t,\cdot_x)](x)=\mathcal{K}[\xi_{H}(\cdot_x-t{W}e_{{z}})](x)$$
  is Lipschitz in space-time $\mathbb{R}^3\times\mathbb{R}_{\geq0}$. Let $C_{Lip}>0$ be the Lipschitz constant. We denote
  the particle trajectory map $\phi_H$ obtained from solving the following ODE system:\\
	$$\frac{d}{dt}\phi_H(t,(t_0,x))=u_H(t,\phi_H(t,(t_0,x)))\quad\mbox{for }t>0\quad\mbox{and } \quad \phi_H(t_0,(t_0,x))=x\in\mathbb{R}^3,$$
	
  Then, we simply observe
$$ \phi_H(t,(t_0,{B}^{{W}t_0}))= {B}^{{W}t}, \quad t,t_0\geq 0.
$$

\item   Fix any $M>0$, consider any 
${\varepsilon_0}\in(0,\min(\varepsilon_1,1/4))$, and set 
\begin{equation*}\label{take_K} 
K:=4\tilde{K}>0,
\end{equation*}
 where $\varepsilon_1, \tilde{K}>0$ are the absolute  constants   from Lemma \ref{lem:no_high_jump}. 
 In the proof, ${\varepsilon_0}>0$ will be taken again  small enough (see \eqref{assum_ep_0_ag}).\\

 Let $\xi_0$ satisfy the assumptions in this lemma with the corresponding solution $\xi(t)$ with some $\varepsilon\in(0,{\varepsilon_0})$ and some $M>0$. Let $u(t):=\mathcal{K}[\xi(t)]$ be the corresponding velocity.\\

	\item  
Fix some $t_0\geq 0$ and $\tau(t_{0})$. 
We shall consider the flow of the unit ball region ${B}^{\tau(t_0)}$ via $u(t,\cdot)$. Our proof is done once we show the following statement:	\ \\

	\noindent \textbf{Bootstrap hypotheses}.     We have, for all $t \in [t_0,t_0+\alp_0 ]$,  \begin{equation} \tag{B1'} 
		\begin{split}
			|\tau(t) - (\tau(t_0) + {W}(t-t_0))| \le K \varep^{1/2}
		\end{split}
	\end{equation}  and 
	 \begin{equation}\tag{B2'}
		\begin{split}
\phi(t,(t_0,{B}^{\tau(t_0)}))\subset {B}^{\tau(t_0)+{W}(t-t_0),\, 1+ (1/2)K\varepsilon^{1/2}}, 
		\end{split}
	\end{equation}  where 
	$\phi$ is the particle trajectory map \eqref{traj}.\\
	
	We note that the hypotheses are trivially valid at $t = t_0$. \ \\
	
\item	We first prove the following claim:
	
	\medskip
	
	\indent  \textbf{Initial claim}. 
There exists	 a constant  $\eta=\eta(M,\varepsilon)>0$  such that,  for any $t\in[t_0,t_0+\eta]$, we have
\begin{equation} \tag{B1-} 
		\begin{split}
			|\tau(t) - (\tau(t_0) + {W}(t-t_0))| \le \frac{1}{2}K \varep^{1/2}
		\end{split}
	\end{equation}  
	and 
	 \begin{equation}\tag{B2-}
		\begin{split}
\phi(t,(t_0,x))\in {B}^{\tau(t_0)+{W}(t-t_0),\, 1+ (1/2)K\varepsilon^{1/2}},\quad
x\in {B}^{\tau(t_0)}.
		\end{split}
	\end{equation} 

	\medskip
	The above claim is essentially done by  Lemma \ref{lem:no_high_jump}. Indeed, the lemma says
	$$	|\tau(t) - \tau(t_0)  | \leq \tilde{K}\varepsilon^{1/2}$$
	whenever $t\in[t_0,t_0+ \tilde{C}\varep^{1/2} ]$. Here, $\tilde{C}=\tilde{C}(M)>0, \tilde{K}>0$ are the constants from Lemma \ref{lem:no_high_jump}. Thus, on this interval,
	$$	|\tau(t) - (\tau(t_0) + {W}(t-t_0))| \leq \tilde{K}\varepsilon^{1/2} +{W}|t-t_0|. 
	$$
We take a constant $\eta>0$ so that
$$\eta\leq\tilde{C}\varepsilon^{1/2}\quad\mbox{and}\quad
{W}\eta\leq \tilde{K}\varepsilon^{1/2}. 
$$ This choice of $\eta$ guarantees that   (B1-) holds for any $t\in[t_0,t_0+\eta]$ due to 
$$	|\tau(t) - (\tau(t_0) + {W}(t-t_0))| \leq 2\tilde{K}\varepsilon^{1/2} 
\leq \frac 1 2 K\varepsilon^{1/2}.   $$
 
For  (B2-), we note that the flow speed is uniformly bounded (see \eqref{unif_speed_h}):
\begin{equation}\label{unif_speed}
\|
u(t)
\|_{L^\infty}\leq C_1\sqrt{M},
\end{equation} where $C_1$ is an absolute constant.
Thus, for any $x_0\in {B}^{\tau(t_0),1}$, we have, for any $t\in\mathbb{R}$,
$$
x:=\phi(t,(t_0,x_0))\in {B}^{\tau(t_0),\, 1+ C_1\sqrt{M}|t-t_0|}.$$ Then, we compute
\begin{equation*}
	\begin{split}
		|x-(\tau(t_0)+{W}(t-t_0))e_{{z}}|&\leq |x-\tau(t_0)e_{{z}}|+{W}|t-t_0|\\
		&\leq  \left(1+ C_1\sqrt{M}|t-t_0|\right) +{W}|t-t_0|\\
		&=  1+ (C_1\sqrt{M}+{W})|t-t_0|.
	\end{split}
\end{equation*} 
Once we make $\eta>0$ smaller than before 
 (if necessary) in order to have
 \begin{equation}\label{take_eta}
 (C_1\sqrt{M}+{W})\eta \leq (7/16)K\varepsilon^{1/2},
 \end{equation}
we obtain   (B2-) for the short interval $[t_0,t_0+\eta]$.\\

\item From now on, we may assume that (B1') and (B2') are valid for $t \in [t_0,t^*]$ with some $t^*>t_0$. The existence of such a $t^*$ is guaranteed by \textbf{Initial claim} (B1-) and (B2-).  We shall prove the following \textit{bootstrap} claim: \ \\

\indent  \textbf{Bootstrap claim}. For each $M$, there exists a small  constant $\alp_0>0$ depending \textit{only} on $M$ such that if $t^* \leq  t_0 + \alp_0 
$, 
 then actually (B1') and (B2') hold for any $t\in[t_0,t^*]$ with $1$ and $1/2$ appearing as the coefficients of $K$ replaced with $1/8$ and ${1/240}$, respectively, \textit{i.e.} we claim, for $t\in[t_0,t^*]$,
\begin{equation} \tag{B1*} 
		\begin{split}
			|\tau(t) - (\tau(t_0) + {W}(t-t_0))| \le \frac{1}{8}K \varep^{1/2}
		\end{split}
	\end{equation}  and 
	 \begin{equation}\tag{B2*}
		\begin{split}
\phi(t,(t_0,{B}^{\tau(t_0)}))\subset {B}^{\tau(t_0)+{W}(t-t_0),\, 1+ ({1/240})K\varepsilon^{1/2}}.
		\end{split}
	\end{equation} 

To verify  \textbf{Bootstrap claim} (B2*), we fix
any $x_0\in {B}^{\tau(t_0)}$ and
compute with $$\phi(t):=\phi(t,(t_0,x_0))$$ that 
\begin{equation}\label{decomp_123}
\begin{split}
	\frac{d}{dt} \phi(t)    = u(t,\phi(t) ) &= u(t,\phi(t) ) - u_H({W}^{-1}{\tau(t)}, \phi(t)) \\
	&\qquad  + u_H ({W}^{-1}{\tau(t)},\phi(t)) - u_H({W}^{-1}{\tau(t_0)}+(t-t_0),\phi(t))  \\
	&\qquad + u_H({W}^{-1}{\tau(t_0)}+(t-t_0),\phi(t))\\
& =:I(t)+II(t)+III(t).
\end{split}
\end{equation}
From the stability assumption  \eqref{assum_lem_stab} and the estimate\eqref{est_vel}, we have 
\begin{equation}\label{est_1}
	\begin{split}
	|I(t)|&
=|\mathcal{K}[\xi(t)](\phi(t)) -\mathcal{K}[\xi_{H}(\cdot_x-\tau(t)e_{{z}})](\phi(t))|\\ &	
	\leq
\|\mathcal{K}[\xi(t)-\xi_{H}(\cdot_x-\tau(t)e_{{z}})]\|_{L^\infty}\\ &	
	\leq C\|r^2\left(\xi(t)-\xi_{H}(\cdot_x-\tau(t)e_{{z}})\right)\|_{L^1}^{1/4}\|\xi(t)-\xi_{H}(\cdot_x-\tau(t)e_{{z}})\|_{L^1}^{1/4}\|\xi(t)-\xi_{H}(\cdot_x-\tau(t)e_{{z}})\|_{L^\infty}^{1/2}\\ &  \leq C_2(M+1)^{1/2} \varep^{1/2},\quad t\geq 0,
	\end{split}
\end{equation} where $C_2>0$ is an absolute constant.\\

For $II(t)$, we use Lipschitz continuity (in space-time) of $u_H$ and the hypotheses (B1') on $[t_0,t^*]$ to get,
for any $t\in[t_0,t^*]$,
\begin{equation}\label{est_2}
	\begin{split}
	|II(t)|&\leq
C_{Lip}	\cdot |
	 {W}^{-1}{\tau(t)} -\left({W}^{-1}{\tau(t_0)}+(t-t_0)\right)|\\
	 &=C_{Lip}{W}^{-1}\cdot |\tau(t) - (\tau(t_0) + {W}(t-t_0))| \le C_{Lip}{W}^{-1}\cdot K \varep^{1/2}.
	\end{split}
\end{equation}
We denote
$$\psi(t):=\phi_H({W}^{-1}{\tau(t_0)}+(t-t_0),({W}^{-1}{\tau(t_0)},x_0)).$$ Then $\psi$ satisfies 
$$\psi(t_0)=x_0\quad\mbox{and}\quad  \psi(t)\in {B}^{\tau(t_0)+{W}(t-t_0),\, 1}\quad\mbox{for any} \quad  t\geq t_0$$ and
\begin{equation*}
	\begin{split}\frac{d}{dt}\psi(t)&=
\frac{d}{dt}\phi_H({W}^{-1}{\tau(t_0)}+(t-t_0),({W}^{-1}{\tau(t_0)},x_0))\\&
=u_H({W}^{-1}{\tau(t_0)}+(t-t_0),\phi_H({W}^{-1}{\tau(t_0)}+(t-t_0),({W}^{-1}{\tau(t_0)},x_0)))\\
&=u_H({W}^{-1}{\tau(t_0)}+(t-t_0),\psi(t)). \end{split}
\end{equation*}
Using the bounds \eqref{est_1}, \eqref{est_2} and comparing the equations for $\phi$ and $\psi$, we see that, for $t\in[t_0,t^*]$, \begin{equation*}
\begin{split}
	\frac{d}{dt}|\phi(t) - \psi(t)| & \le |u_H({W}^{-1}{\tau(t_0)}+(t-t_0),\phi(t)) -u_H({W}^{-1}{\tau(t_0)}+(t-t_0),\psi(t)) |  \\
	&\qquad\qquad +C_2(M+1)^{1/2} \varep^{1/2}+C_{Lip}{W}^{-1}\cdot K \varep^{1/2} \\
	& \le C_{Lip}|\phi(t)-\psi(t)| + C_3\cdot(M+1) \varep^{1/2},
\end{split}
\end{equation*} where $C_3>0$ is an absolute constant.
With Gronwall's inequality, we deduce for $t\in[t_0,t^*]\subset[t_0,t_0+\alpha_0 ]$ that \begin{equation*}
\begin{split}
	|\phi(t)-\psi(t)|& \le e^{C_{Lip}(t^*-t_0)}\cdot\int_{t_0}^{t^*}    
C_3(1+M)\varep^{1/2} ds. \\
	& \le   \left( e^{ C_{Lip}\alp_0  }C_3(1+M)\alp_0\right)\cdot  \varep^{1/2}.  
\end{split}
\end{equation*} We take $\alp_0=\alp_0(M)>0$ small enough so that
$$
\left(e^{ C_{Lip}\alp_0  }C_3(1+M)  \alp_0\right)\leq \frac 1 {240} K.
$$
Since we know $\psi(t)\in {B}^{\tau(t_0)+{W}(t-t_0),\, 1}$, the above estimate shows
	 \begin{equation*} 
		\begin{split}
\phi(t)\in {B}^{\tau(t_0)+{W}(t-t_0),\, 1+ ({1/240})K\varepsilon^{1/2}},
		\end{split}
	\end{equation*} which is \textbf{Bootstrap claim} (B2*) on $[t_0,t^*]$ whenever $t^*\leq t_0+\alpha_0$.\\

	To prove (B1*) on $[t_0,t^*]$ when $t^*\leq t_0+\alpha_0$, we
	denote   $$A_t:=\phi(t,(t_0,{B}^{\tau(t_0)})),\quad  t\geq t_0,
	$$ and decompose$$
	\xi(t,x)=\xi(t,x){\mathbf{1}}_{A_t}(x)+\xi(t,x){\mathbf{1}}_{\mathbb{R}^3\setminus A_t}(x)=:\Omega^1(t,x)+\Omega^2(t,x).
	$$

As before, we 	denote, for $\tau\in\mathbb{R}$, 
$$\xi_H^{\tau}:=\xi_H(\cdot_x-\tau e_{{z}}).$$ Then we note $$\|\Omega^2(t)\|_{L^1}=\|\Omega^2(t_0)\|_{L^1}\leq
	\|\xi(t_0)-\xi_H^{\tau(t_0)}\|_{L^1({\mathbb{R}^3\setminus A_{t_0}})}\leq
	\|\xi(t_0)-\xi_H^{\tau(t_0)}\|_{L^1}
	\leq  \varepsilon$$ since 	 $\xi_H^{\tau(t_0)}$ is supported in $A_{t_0}={B}^{\tau(t_0)}$. \\
	
Towards	a contradiction, let us assume that (B1*) on $[t_0,t^*]$ with   $t^*\leq t_0+\alpha_0$ fails, \textit{i.e.}  there is some $t'\in[t_0,t^*]$ satisfying
\begin{equation*}  
		\begin{split}
			|\tau(t') - (\tau(t_0) + {W}(t'-t_0))| > \frac{1}{8}K \varep^{1/2}.
		\end{split}
	\end{equation*} We may assume $$\tau(t') - (\tau(t_0) + {W}(t'-t_0)) > \frac{1}{8}K \varep^{1/2},$$ since the other case can be considered similarly.
Because we already obtained   (B2*) on $[t_0,t^*]$ with   $t^*\leq t_0+\alpha_0$,	 we recall  
$$A_{t'}\subset  {B}^{\tau(t_0)+{W}(t'-t_0),\, 1+ ({1/240})K\varepsilon^{1/2}}
$$  so we have
			 \begin{equation} \label{decomp}
		\begin{split}
		\|\xi(t')-\xi_H^{\tau(t')}\|_{L^1}&\geq 
		\|\Omega^1(t')-\xi_H^{\tau(t')}\|_{L^1}-\|\Omega^2(t')\|_{L^1}\\
		&\geq \|\Omega^1(t')-\xi_H^{\tau(t')}\|_{L^1(\mathbb{R}^3\setminus A_{t'})}-\varepsilon= \| \xi_H^{\tau(t')}\|_{L^1(\mathbb{R}^3\setminus A_{t'})}-\varepsilon
 		\end{split}
	\end{equation}	and 
			 \begin{equation*} 
		\begin{split}
& \| \xi_H^{\tau(t')}\|_{L^1(\mathbb{R}^3\setminus A_{t'})}\geq 
 \| \xi_H^{\tau(t')}\|_{L^1(\mathbb{R}^3\setminus  {B}^{\tau(t_0)+{W}(t'-t_0),\, 1+ ({1/240})K\varepsilon^{1/2}})} 	
 	 \\&\quad \geq
  \| \xi_H^{\tau(t')}\|_{L^1(\mathbb{R}^3\setminus
{B}^{\tau(t_0)+{W}(t'-t_0),\, 1}
  )}-|{B}^{\tau(t_0)+{W}(t'-t_0),\, 1+ ({1/240})K\varepsilon^{1/2}}\triangle {B}^{\tau(t_0)+{W}(t'-t_0),\, 1 }|  \\
  &\quad =:I(t')-II(t').
 		\end{split}
	\end{equation*} 
	We assume $\varepsilon_0>0$ small enough (if necessary) to guarantee
	\begin{equation}\label{assum_ep_0_ag}
		 K\varepsilon_0^{1/2} \leq 2	
		\quad\mbox{and}\quad
		 	\varepsilon_0^{1/2}\leq \frac{1}{24}K.		
	\end{equation}
	For $II(t')$, we use 
\eqref{kappa_est} to get	
	simply have
	$$
	|II(t')|\leq 30 \cdot ({1/240})K\varepsilon^{1/2}=   ({1/8})K\varepsilon^{1/2}.
	$$ Here
we used the assumption \eqref{assum_ep_0_ag}.	
 	For $I(t')$, we estimate 
	 \begin{equation*} 
		\begin{split}
I(t')&=   |B^{\tau(t')}\setminus
{B}^{\tau(t_0)+{W}(t'-t_0),\, 1}
  ) |=\frac{1}{2}
   |B^{\tau(t')}\triangle
{B}^{\tau(t_0)+{W}(t'-t_0),\, 1}
  ) |\\
  &=\frac 1 2 f\left(\tau(t')-\left(
 \tau(t_0)+{W}(t'-t_0)\right)\right)\\
  		&\geq 
 		\frac 1 2 \cdot\frac 4 3 \pi\cdot \left(\tau(t')-\left(
 \tau(t_0)+{W}(t'-t_0)\right)\right),
 		\end{split}
	\end{equation*} 
 where the last inequality follows from \eqref{claim_hill_1}	 thanks to the assumption \eqref{assum_ep_0_ag}. Thus, we obtain
  \begin{equation*} 
		\begin{split}
  \| \xi_H^{\tau(t')}\|_{L^1(\mathbb{R}^3\setminus A_{t'})} 	&\geq I(t')-II(t')\\
  &\geq 	\frac 1 2 \cdot\frac 4 3 \pi\cdot \left(\tau(t')-\left(
 \tau(t_0)+{W}(t'-t_0)\right)\right) - ({1/8})K\varepsilon^{1/2}\\
 &\geq \frac 1 2 \cdot\frac 4 3 \pi\cdot ({1/8})K\varepsilon^{1/2} - ({1/8})K\varepsilon^{1/2}\geq ({1/8})K\varepsilon^{1/2}.
 		\end{split}
	\end{equation*} 
	 Going back to \eqref{decomp}, we conclude	 
	 	 \begin{equation*} 
		\begin{split}
		\|\xi(t')-\xi_H^{\tau(t')}\|_{L^1}&\geq 
	\| \xi_H^{\tau(t')}\|_{L^1(\mathbb{R}^3\setminus A_{t'})}-\varepsilon\\
	&\geq  ({1/8})K\varepsilon^{1/2} -\varepsilon\geq 2 \varepsilon,
 		\end{split}
	\end{equation*} 
where we used \eqref{assum_ep_0_ag} again for the last inequality,	
	which violates the condition \eqref{assum_lem_stab}. Hence, we obtained (B1*) on $[t_0,t^*]$.

\item We are in a position to finish the proof since we can slightly extend the time interval $[t_0,t^*]$ on which (B1') and (B2') hold (and therefore also (B1*) and (B2*)) since 
		 we can apply \textbf{Initial claim} (B1-) 
		 to the time interval $[t^*,t^*+\eta]$, where
the constant $\eta=\eta(M,\varepsilon)>0$ was chosen in step 4.
 More precisely,  \textbf{Initial claim} (B1-) implies for any $t\in[t^*,t^*+\eta]$,
	\begin{equation}  
		\begin{split}
			|\tau(t) - (\tau(t^*) + {W}(t-t^*))| \le \frac 1 2 K \varep^{1/2}.
		\end{split}
	\end{equation}  
We add the above estimate into (B1*) (for $t=t^*$) in order to obtain, for $t\in[t^*,t^*+\eta]$,
$$ 
|\tau(t) - (\tau(t_0) + {W}(t-t_0))| \le \frac{1}{8}K \varep^{1/2}+\frac 1 2 K \varep^{1/2}\leq  K \varep^{1/2}.
$$ Thus we have (B1') on the extended interval $[t_0,t^*+\eta]$. \\

To extend (B2') up to $[t_0,t^*+\eta]$, we recall (B2*) for $t=t^*$:
$$\phi(t^*,(t_0,{B}^{\tau(t_0)}))\subset {B}^{\tau(t_0)+{W}(t^*-t_0),\, 1+ ({1/240})K\varepsilon^{1/2}}.$$ Thanks to the uniform bound \eqref{unif_speed} of the flow speed, we have, 
for  $t\in[t^*, t^*+\eta]$,
  	 \begin{equation} \label{comp_short_eat}
		\begin{split}
\phi(t,(t_0,{B}^{\tau(t_0)}))
&\subset \phi(t,(t^*,
 {B}^{\tau(t_0)+{W}(t^*-t_0),\, 1+ ({1/240})K\varepsilon^{1/2}}))\\
 &\subset   {B}^{\tau(t_0)+{W}(t^*-t_0),\, 1+ ({1/240})K\varepsilon^{1/2}+C_1\sqrt{M}\eta}
		\end{split}
	\end{equation}
	
	We claim   
	\begin{equation}\label{temp_claim}
		{B}^{\tau(t_0)+{W}(t^*-t_0),\, 1+ ({1/240})K\varepsilon^{1/2}+C_1\sqrt{M}\eta}\subset {B}^{\tau(t_0)+{W}(t-t_0),\, 1+ (1/2)K\varepsilon^{1/2}}.
	\end{equation}
Indeed, for any $y\in {B}^{\tau(t_0)+{W}(t^*-t_0),\, 1+ ({1/240})K\varepsilon^{1/2}+C_1\sqrt{M}\eta}$, 	we compute
\begin{equation*} 
		\begin{split}
 |y-\left(\tau(t_0)+{W}(t-t_0)\right)e_{{z}}|&\leq
 |y-\left(\tau(t_0)+{W}(t^*-t_0)\right)e_{{z}}|+{W}|t^*-t|\\ 
  &\leq  1+ ({1/240})K\varepsilon^{1/2}+C_1\sqrt{M}\eta+{W}\eta.
		\end{split}
	\end{equation*}	
		Since we assumed the condition \eqref{take_eta}, we obtained the above claim \eqref{temp_claim}.		
 Together with
\eqref{comp_short_eat}, it implies,
	  for $t\in[t^*, t^*+\eta]$,
	$$\phi(t,(t_0,{B}^{\tau(t_0)}))\subset  {B}^{\tau(t_0)+{W}(t-t_0),\, 1+ (1/2)K\varepsilon^{1/2}},$$ which gives (B2') on the extended interval $[t_0,t^*+\eta]$.\\
 
	 Next, we can obtain \textbf{Bootstrap claim} (B1*) and (B2*) on the extended interval $[t_0,t^*+\eta]$. Then we simply repeat the same process above to get that (B1*) and (B2*) are actually valid on $[t_0,t^*+k\eta]$ for any $k\in\mathbb{N}$ until (B1') and (B2') are covered on the \textit{full} interval $[t_0,t_0+\alpha_0]$. This finishes the proof. \qedhere  \end{enumerate} 
\end{proof}



\subsection{Proof of Theorem  \ref{thm:B}} 
\begin{proof}

(I) The orbital stability  can be found in    \cite[Theorem 1.2]{Choi2020} (or just see Theorem \ref{thm_hill_gen}). Here, we can simply take $\tau(0)=0$ (by assuming $\delta(\varepsilon)\leq \varepsilon$ if necessary).\\

(II) The estimate for the shift function $\tau$  is a direct consequence of Proposition \ref{prop:travel-speed_hill}. Indeed, 
 a simple summation  
of the estimate \eqref{eq:travel-est_hill} 
  gives \begin{equation*}
	\begin{split}
		|\tau(t) - \tau(0) -{W}t| \le (K/\alpha_0) (t+\alpha_0)\varepsilon^{1/2}, \qquad t\ge 0. \qedhere
	\end{split}
\end{equation*}  
\end{proof}

 \subsection{Finite-time growth in vorticity (Proof of Corollary \ref{cor:vor:grow})}
In order to obtain growth of the vorticity, we first prove a finite-time stability result using Theorem \ref{thm:B}. 
 \begin{lemma}\label{lem_finite_gro} 
 For each $M>0$, $T>0$, and  $\lambda_0\in(0,1)$, there exists $\varepsilon_2 = \varep_2(M,T,\lambda_0)>0$ such that 
for any non-negative axi-symmetric $\xi_{0}$ satisfying $\xi_{0}, r\xi_{0} \in L^{\infty}(\bbR^3)$, and
$\|\xi_0\|_{L^\infty}\leq M$,
if  there is a shift function $\tau(\cdot):[0,\infty)\to\mathbb{R}$ satisfying
\begin{align*}
	\sup_{t\geq 0}\left\{	\|\xi(\cdot+\tau(t) e_{z},t)-\xi_H\|_{L^1\cap L^2(\mathbb{R}^3)}
		+\|r^2(\xi(\cdot+\tau(t) e_{z},t)-\xi_H)\|_{L^1(\mathbb{R}^3)} \right\}\leq \varepsilon
	\end{align*}  for some $\varepsilon\in(0,\varepsilon_2)$,
	then we have, for any $t\in[0,T]$,
	 	 \begin{equation}\label{lem_finite_concl}
		\begin{split}
		&	|\tau(t)-{W}t| \le \lambda_0\quad\mbox{and}\quad \phi(t,(0,{S}^{0}))\subset {S}^{{W}t,\,  \lambda_0},
		\end{split}
	\end{equation}  where 
	$\phi$ is the particle trajectory map  \eqref{traj},
	 and $S^{\tau,a}$ are spherical shells defined by ($\tau\in\mathbb{R}, a\in[0,1]$) 
$$ {{S}^{\tau,a}}:=\{x\in\mathbb{R}^3\,|\, 1-a\leq  |x-\tau e_{{z}}|\leq 1+a\},\quad 
 {{S}^{\tau}}:= {{S}^{\tau,0}}.$$
 \end{lemma}
 \begin{proof}
 First, 
we fix any    $M>0$,   $T>0$, and   $\lambda_0\in(0,1)$. Let $\lambda_1\in(0,\lambda_0)$ be a constant which will be chosen small later (see \eqref{lam_1}). 
  We assume  $\varepsilon_2\in(0,\varepsilon_0)$ small enough to satisfy
\begin{equation}\label{ep_2}
  C_M {\varep_2^{\frac12}(T +1)}\leq \lambda_1, 
\end{equation}  where $\varepsilon_0, C_M$ are the constants from Theorem \ref{thm:B}.
Then
 the estimate \eqref{main_est} on a shift $\tau(t)$ of Theorem \ref{thm:B} becomes \begin{equation}\label{est_tau_lam}
  			|\tau(t)-{W}t| \le C_M {\varep^{\frac12}(t +1)}
  			\leq \lambda_1\leq \lambda_0,\quad t\in[0,T],
\end{equation} which gives the first estimate \eqref{lem_finite_concl}.\\

For the second estimate,  we proceed as in the proof of \textbf{Bootstrap claim} (B2*) in Proposition \ref{prop:travel-speed_hill}. Indeed, we fix
any $x_0\in S^0$, denote
  $$\phi(t):=\phi(t,(t_0,x_0)),$$ and use the same decomposition \eqref{decomp_123} (with $t_0=0$):
\begin{equation*}
\begin{split}
	\frac{d}{dt} \phi(t)    = u(t,\phi(t) ) &= u(t,\phi(t) ) - u_H({W}^{-1}{\tau(t)}, \phi(t)) \\
	&\qquad  + u_H ({W}^{-1}{\tau(t)},\phi(t)) - u_H(t,\phi(t))  \\
	&\qquad + u_H(t,\phi(t))\\
& =:I(t)+II(t)+III(t).
\end{split}
\end{equation*} 
Then as in \eqref{est_1}, we have 
\begin{equation*}
	\begin{split}
	|I(t)|&
  \leq C_2(M+1)^{1/2} \varep^{1/2},\quad t\geq 0,
	\end{split}
\end{equation*} where $C_2>0$ is an absolute constant.\\

For $II(t)$, as in \eqref{est_2}, we obtain,
for any $t\in[0,T]$,
\begin{equation*}
	\begin{split}
	|II(t)|&\leq
C_{Lip}	\cdot |
	 {W}^{-1}{\tau(t)} -t|\\
	 &=C_{Lip}{W}^{-1}\cdot |\tau(t) -{W}t| \le C_{Lip}{W}^{-1}\cdot \lambda_1,	\end{split}
\end{equation*}
where $C_{Lip}>0$ is the (space-time) Lipschitz constant of the velocity for Hill's vortex $\xi_H$.
 Here we used \eqref{est_tau_lam} in the last inequality.\\

Denoting
$$\psi(t):=\phi_H(t,(0,x_0)),$$ it satisfies 
$$\psi(t_0)=x_0\quad\mbox{and}\quad  \psi(t)\in {S}^{{W}t,\, 0}\quad\mbox{for any} \quad  t\geq 0$$ and
\begin{equation*}
	\begin{split}\frac{d}{dt}\psi(t)&=u_H(t,\psi(t)). \end{split}
\end{equation*}
The above bounds implies, for $t\in[0,T]$, \begin{equation*}
\begin{split}
	\frac{d}{dt}|\phi(t) - \psi(t)| & \le |u_H(t,\phi(t)) -u_H(t,\psi(t)) |  +C_2(M+1)^{1/2} \varep^{1/2}+C_{Lip}{W}^{-1}\cdot \lambda_1\\
	 & \le  C_{Lip}|\phi(t)-\psi(t)|  +\underbrace{\left( \frac{C_2(M+1)^{1/2}}{C_M(T+1)} +C_{Lip}{W}^{-1}\right)}_{=:\hat{C}} {\lambda_1},
\end{split}
\end{equation*} where we have used \eqref{ep_2}. With Gronwall's inequality, we deduce for $t\in[0,T]$ that \begin{equation*}
\begin{split}
	|\phi(t)-\psi(t)|& \le e^{C_{Lip}T}\cdot\int_{0}^{T}    
\hat{C}\lambda_1 ds 
 \le     e^{ C_{Lip}T  }\hat{C} T \lambda_1.
\end{split}
\end{equation*}  
We assume $\lambda_1\in(0,\lambda_0)$ small enough so that
\begin{equation}\label{lam_1}
  e^{ C_{Lip}T  }\hat{C} T \lambda_1\leq \lambda_0,
\end{equation} which implies 
\begin{equation*}
\begin{split}
	|\phi(t)-\psi(t)| \leq  \lambda_0, \quad t\in[0,T].
\end{split}
\end{equation*}  
Since $\psi(t)\in {S}^{{W} t,\, 0}$, the above estimate shows that 
	 \begin{equation*} 
		\begin{split}
\phi(t)\in {S}^{{W} t,\, \lambda_0},\quad t\in[0,T]. \qedhere 
		\end{split}
	\end{equation*}  \end{proof}
 Now we are ready to prove finite time growth in vorticity. 
\begin{proof}[Proof of Corollary \ref{cor:vor:grow}]

 \begin{enumerate}

  \item  We recall 
   the explicit expression 
\eqref{vel_explicit}   
   of the Lipschitz velocity $u_H$ of Hill's vortex $\xi_H$. In particular, the information 
$$u^r_H(x)|_{r\geq 0,\,z=1}   =\frac{3}{2} {W}\frac{r}{(r^2+1)^{5/2}}$$
guarantees that there is a constant  $a \in(0,1/10)$ such that
  if we set a cylinder
\begin{equation}
F:=\{
x\in\mathbb{R}^3\,|\, 1-a
\leq z\leq 1+a
 \quad\mbox{and} \quad 0\leq r \leq \frac 1 {10}
\}
\end{equation}  then the velocity on this cylinder satisfies
 \begin{equation}
\frac 1 {5} {W}r\leq  u^r_H(x)\leq 5 {W}r,\quad x\in F.
 \end{equation}
 
 \item 
For each $\tau\in\mathbb{R}, \lambda\in[0,1]$, we denote 
\begin{equation*}\label{E_la}
E^{\tau,\lambda}
:=
S^{\tau,\lambda}\cap \{ 0\leq r\leq \frac 1 {10},\quad z\geq \tau\},\quad 
E^{\tau}:=E^{\tau,0},
\end{equation*}  where spherical shells ${{S}^{\tau,\lambda}}$ are defined by
($\tau\in\mathbb{R}$)
$$ {{S}^{\tau,\lambda}}:=\{x\in\mathbb{R}^3\,|\, 1-\lambda\leq  |x-\tau e_{{z}}|\leq 1+\lambda\},\quad 
 {{S}^{\tau}}:= {{S}^{\tau,0}}.$$
 We can easily verify that there is $\lambda_0>0$ (small depending on $a$) such that,
 
  $$E^{\tau,\lambda_0} \subset F\quad \mbox{for any}\quad |\tau|\leq \lambda_0.$$ It guarantees,
   for any $ x\in E^{\tau,\lambda_0}$ with $ |\tau|\leq \lambda_0$,
 \begin{equation}\label{est_ur_e}
\frac 1 {5} {W}r\leq  u^r_H(x)\leq 5 {W}r.
 \end{equation}
 \item Let $L> 1$ and set 
$$r_0:=\frac 1 {10 L},\quad M:=\frac{2}{r_0},\quad c:=({W}/10)^{-1}, \quad T:=c\log(L).$$ 	We fix a point $x_0=(x_0^1,x_0^2,x_0^3)\in E^0\subset  S^0\subset \mathbb{R}^3$ satisfying 
$ |(x_0^1,x_0^2,0)|=r_0$ (\textit{i.e.} $r(x_0)=r_0$).\\

By applying Lemma \ref{lem_finite_gro} for the above constants $M, T, \lambda_0$, we take 
$\varepsilon_2>0$ from  Lemma \ref{lem_finite_gro}. Let $\varepsilon\in(0,\varepsilon_2)$ which will be taken small later (see \eqref{small_ep_finite}). Consider $\delta=\delta(\varepsilon)>0$ from    Theorem \ref{thm:B}.\\

Now we take a  $C^\infty$--smooth initial data $\xi_{0}\in L^{\infty}(\bbR^3)$
which is compactly supported, axi-symmetric,  non-negative,  
$\|\xi_0\|_{L^\infty}\leq M$, and 
 \begin{align*}
		\|\xi_0-\xi_H\|_{L^1\cap L^2(\mathbb{R}^3)}
		+\|r^2(\xi_0-\xi_H)\|_{L^1(\mathbb{R}^3)} 
		\leq \delta.
	\end{align*} 
	 	 In addition, we assume 
	\begin{equation}
	\xi_0(x_0)=\frac M 2 
	\end{equation} and
	$$\|\omega_0\|_{L^\infty(\mathbb{R}^3)}\leq 1
$$
	 (by recalling the relation $\omega=\omega^\theta e_\theta= r \xi e_\theta$).
	
\item By 
Theorem \ref{thm:B},   there is a shift function $\tau(\cdot):[0,\infty)\to\mathbb{R}$ satisfying
\begin{align*}
	\sup_{t\geq 0}\left\{	\|\xi(\cdot+\tau(t) e_{z},t)-\xi_H\|_{L^1\cap L^2(\mathbb{R}^3)}
		+\|r^2(\xi(\cdot+\tau(t) e_{z},t)-\xi_H)\|_{L^1(\mathbb{R}^3)} \right\}\leq \varepsilon.
	\end{align*}  
	Together with the estimate \eqref{est_vel}, the above (orbital) stability implies
\begin{equation}\label{est_3}
	\begin{split}
&\| u(t) -\mathcal{K}[\xi_{H}(\cdot_x-\tau(t)e_{{z}})]\|_{L^\infty}= 
\|\mathcal{K}[\xi(t)-\xi_{H}(\cdot_x-\tau(t)e_{{z}})]\|_{L^\infty}\\ &	\quad
	\leq C\|r^2\left(\xi(t)-\xi_{H}(\cdot_x-\tau(t)e_{{z}})\right)\|_{L^1}^{1/4}\|\xi(t)-\xi_{H}(\cdot_x-\tau(t)e_{{z}})\|_{L^1}^{1/4}\|\xi(t)-\xi_{H}(\cdot_x-\tau(t)e_{{z}})\|_{L^\infty}^{1/2}\\ & \quad \leq C_2(M+1)^{1/2} \varep^{1/2},\quad t\geq 0,
	\end{split}
\end{equation} where $C_2>0$ is an absolute constant.\\

From now on, we assume $\varepsilon\in(0,\varepsilon_2)$ small enough to have
\begin{equation}\label{small_ep_finite}
 C_2(M+1)^{1/2} \varep^{1/2}\leq \frac 1 {10} {W} r_0.
\end{equation}

	On the other hand,  \eqref{lem_finite_concl} from Lemma \ref{lem_finite_gro} implies, for any $t\in[0,T]$,
		 	 \begin{equation}\label{S_est}
		\begin{split}
|\tau(t)-{W}t| \le \lambda_0 \quad\mbox{and}\quad \phi(t):=\phi(t,(0,x_0))\in {S}^{{W}t,\,  \lambda_0}.
		\end{split}
	\end{equation} 
	\item We observe, for $t\in[0,T]$, that as long as $\phi(t)$ stays
	in $E^{t{W},\lambda_0}$, 	
	we can use the estimate \eqref{est_ur_e} to obtain
	\begin{equation}\label{phi_r_est_finite}
\phi^r(t) \mbox{ is increasing in } t \quad  \mbox{and}\quad
	  \phi^r(t) \geq r_0\exp(  c^{-1} t).
	\end{equation}
Indeed, 	since
$$
\phi(t)\in E^{t{W},\lambda_0}\quad\Rightarrow\quad  \phi(t)-\tau(t)e_z
\in E^{t{W}-\tau(t),\lambda_0}
	$$
we have, by using \eqref{S_est}, \eqref{est_ur_e}, \eqref{est_3}, \eqref{small_ep_finite},
	 	 \begin{equation*}
		\begin{split}
 u^r(t,\phi(t))&\geq u_H^r(\phi(t)-\tau(t)e_z)-| u^r(t,\phi(t))-u_H^r(\phi(t)-\tau(t)e_z)|\\
& \geq  
\frac 1 5 {W} \phi^r(t)
-C_2(M+1)^{1/2} \varep^{1/2}
\geq \frac 1 {5} {W} \phi^r(t)-\frac 1 {10} {W} \phi^r(0).
		\end{split}
	\end{equation*} 
This estimate implies for each $t'\in(0,T]$ that 
$$
\phi^r(t)\geq \frac 1 2 \phi^r(0)\quad\mbox{on}\quad t\in[0,t'] \qquad
\Rightarrow\qquad  u^r(t,\phi(t))\geq 
0
\quad\mbox{on}\quad t\in[0,t'] 
$$	 while the converse  is also true.
	 Thus we obtain the first part (monotonicity) of \eqref{phi_r_est_finite} as long as $\phi(t)$ stays
	in $E^{t{W},\lambda_0}$ (\textit{e.g.} by a  bootstrap argument). As a result, we obtain
		 	 \begin{equation*}
		\begin{split}
 u^r(t,\phi(t))& 
\geq \frac 1 {5} {W} \phi^r(t)-\frac 1 {10} {W} \phi^r(0)\geq 
\frac 1 {10} {W} \phi^r(t),
		\end{split}
	\end{equation*} which implies the second part of \eqref{phi_r_est_finite}.
	
	
\item 	 Denote $[0,T^*]$($\subset[0,T]$) be the maximal time interval 	satisfying$$\phi(t) \in E^{t{W},\lambda_0}\quad \mbox{for all}\quad t\in[0,T^*].$$	Then it is obvious that $T^*>0$ by continuity, and there are two cases: either $T^*=T$ or $T^*<T$. When the first case occurs, we obtain
	$$  \phi^r(T) \geq r_0\exp(  c^{-1} T)=\frac 1{10}.$$ Thus, for this case $T^*=T$,  we conclude
	$$\|\omega(T)\|_{L^\infty}\geq |\omega(T,\phi(T))|=|\phi^r(T)\xi(T,\phi(T))|\geq \frac 1 {10}|\xi_0(x_0)|=\frac M {20}=L.$$
Lastly, let's suppose that $T^*<T$ holds. Thanks to 
\eqref{S_est} and continuity of $\phi$, we see that $\phi(t)$ at $t = T^*$ should lie on  the ``upper'' boundary of the set
$E^{T^*{W},\lambda_0}$, i.e. 
	$$\phi(T^*)\in  \{r=\frac 1 {10}\}.$$
 As before, it implies  \begin{equation*}
		\begin{split}
			\|\omega(T^*)\|_{L^\infty}\geq L. \qedhere 
		\end{split}
	\end{equation*}
	
 \end{enumerate} 
\end{proof}

 \color{black}
 
 \begin{figure}\centering
 	 \includegraphics{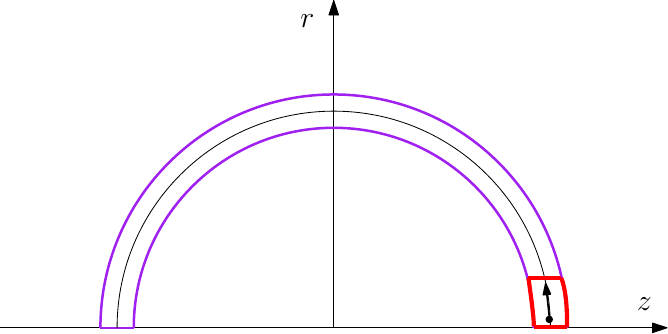}  
 	\caption{A schematic diagram for the proof of Corollary \ref{cor:vor:grow}. Two thickened regions represent $S^{\tau,\lmb}$ (large) and $E^{\tau,\lmb}$ (small), respectively. In the frame moving with velocity $We^z$, fluid particles in the interior of $E^{\tau,\lmb}$ ``climb up'' in the $r$-direction, leading to growth of the vorticity in $L^\infty$.}
 	\label{fig1}
 \end{figure}
\section{Dynamics of the perturbed Hill's vortex}
 
\subsection{Trajectories in Hill's vortex}
To motivate the statements of our main theorems below, let us compute the particle trajectories $\phi_H$ associated with Hill's vortex. For simplicity, we consider trajectories in the moving frame with velocity $We_{z}$, so that the boundary of the unit ball $\partial B$ is invariant by the flow. 
The flow $\phi^*$ defined by \begin{equation*}
	\begin{split}
		\frac{d}{dt}\phi^* (t, x) = v(t,\phi^*(t,x)) , \qquad \phi^*(0,x)=x
	\end{split}
\end{equation*} with $v:= u_H - We_z$ is related to the original flow simply by $\phi_H(t,x)=\phi^*(t,x)+Wte_z$ (see \eqref{vel_explicit} for $u_H$). 
With some abuse of notation, we shall denote 
\begin{equation*}
	\begin{split}
		v(z) := u^{z}_{H}(r=0,z)-W = \begin{cases}
			\frac{3W}{2}(1- {|z|^2}), \quad & |z|\le 1, \\ 
			\frac{W}{ {|z|^{3}}}(1- {|z|^{3}}), \quad & |z| > 1
		\end{cases}
	\end{split}
\end{equation*} and \begin{equation}\label{eq:ODE-axis}
\begin{split}
	\frac{d}{dt} z(t) = v(z(t)), \quad z(0)=z_{0} .
\end{split}
\end{equation}  We consider three cases.

\medskip

\noindent \textit{Case 1:} $|z_{0}|<1$. In this case, the solution to \eqref{eq:ODE-axis} is explicitly given by \begin{equation*}
\begin{split}
 {	z(t) = \frac{e^{3Wt}-e^{2c}}{e^{3Wt}+e^{2c}}, \quad c:=\frac12 \ln \frac{1-z_0}{1+z_0}\in(-\infty,\infty).}
\end{split}
\end{equation*} Note that $z(t)$ is monotonically increasing and $\lim_{t\to\infty} z(t) =1$ (see Theorem \ref{thm:fila-patch-2} case 1). 

\medskip

\noindent \textit{Case 2:} $z_{0}>1$. Note that when $z>1$, $v(z)<0$ again. We have that $z(t)$ is strictly decreasing with time and converges to 1 at an exponential rate as $t\rightarrow \infty$  (see Theorem \ref{thm:fila-patch-2} case 2). 

\medskip

\noindent \textit{Case 3:} $z_{0}<-1$. Whenever $z<-1$, we have $v(z)<0$ and therefore $z(t)$ is strictly decreasing with time. When $|z|$ becomes large, $v(z) \simeq -1$ and therefore $z(t)$ decreases asymptotically \textit{linearly} in time with rate $W$. This case is responsible for the formation of a long tail (see Theorem \ref{thm:fila-patch}).
  
\subsection{Patch solutions}
 
We state the main results of this section regarding patches, which show that for patch data sufficiently close to the Hill's vortex, the particle trajectories of the associated solution lying on the symmetry axis follow (very closely) the corresponding trajectories for the Hill's vortex. More specifically, we prove that, for particles initially ``behind'' the vortex, linear in time separation of the trajectory from the vortex core;  this is the content of Theorem \ref{thm:fila-patch}. All the other particles starting from the axis approach the front edge of the moving vortex, but in this case we lose control once the trajectory becomes too close to the edge; see the statement of Theorem \ref{thm:fila-patch-2} below. In particular, these phenomena described by Theorem  \ref{thm:fila-patch} for a tail elongating and the case 1 of Theorem \ref{thm:fila-patch-2}  for an infiltration from interior point  were confirmed numerically by contour dynamics (Pozrikidis \cite{pozrikidis_1986}) for 
a prolate spheroid and an oblate spheroid, respectively.

\begin{theorem}\label{thm:fila-patch}
	For any $0< \eta < 1$, there exists a constant $\varepsilon_{0}>0$ such that the following holds for any $0< \varepsilon \le \varepsilon_{0}$: Let a patch-type initial data $\xi_0={\mathbf{1}}_{A_{0}}$ satisfy the assumptions of Theorem \ref{thm:B}, and let $\xi(t)={\mathbf{1}}_{A(t)}$ and $\phi(t)$ be the corresponding solution  and the particle trajectory map associated with the  data $\xi_{0} $. Then for $x_0 = (0,z_{0}) =z_0e_z$ with $z_0 < - 1 - \eta$,  for all $t \ge 0$ we have that $\phi^{z}(t,x_{0})- Wt$ is decreasing in time and \begin{equation*}
			\begin{split}
				\phi^{z}(t,x_{0}) - (W - \mu)t < -1 - {\eta} 
			\end{split}
		\end{equation*} for some $\mu>0$ depending only on the choice of $\eta\in(0,1)$. 
\end{theorem} 
 
\begin{theorem}\label{thm:fila-patch-2} For any $\eta>0$, there exists a constant 
$\varepsilon_0>0$ such that under the same assumptions from Theorem \ref{thm:fila-patch}, we have for $x_0 = (0,z_{0})$ that 
	\begin{enumerate}
 		\item (Interior) if $ -1+ \eta <  z_0 < 1 - \eta$, there exists
 		 ${T=T(\varep,\eta)}\in(0,\infty)$ such that \begin{equation}\label{est_infil}
			\begin{split}
				\phi^{z}(T,x_{0}) - \tau(T) \geq  1 - C_{2} \varep^{\frac12} 
			\end{split}
		\end{equation} and for $0 \le t \le T $, we have that $\phi^{z}(t,x_{0}) -  (W + 2C_*\varep^{\frac12} )t$ is increasing in time. 
		\item (Ahead) if $1 + \eta < z_0$, there exists
 		 $T=T(\varep, \eta)\in(0,\infty)$ such that  \begin{equation*}
			\begin{split}
				\phi^{z}(T,x_{0}) - \tau(T) \le   1 + C_{2} \varep^{\frac12} 
			\end{split}
		\end{equation*} and  for $0 \le t \le T$, we have that $\phi^{z}(t,x_{0}) - (W - 2C_*\varep^{\frac12}  )t$ is decreasing in time.
	\end{enumerate} 	
	 In both cases, $C_2>0$ is an absolute constant, and $\phi^{z}$ satisfies \begin{equation}\label{bothclaim}
	\begin{split}
		\left| \phi^{z}(t,x_{0}) - (\tau(t) +1 )\right| \le 2C_{2} \varep^{\frac12}, \qquad t \ge T. 
	\end{split}
\end{equation}  
	Here, $C_*>0$ is the absolute constant from \eqref{main_est} with the choice $M=1$. 
	 
\end{theorem}

Before proceeding to the proof of the above theorems, we illustrate their consequences.  First,
 we obtain linear growth in perimeter (Corollary \ref{cor:peri_growth}).
\begin{proof}[Proof of Corollary \ref{cor:peri_growth}]
 
	Given any $\dlt'>0$, we can simply take some axi-symmetric open set $A_{0}$ 
	 which contains a line
segment $$ 
L:=\{r=0,\quad z\in[-1-\dlt'/2,1]\}.
$$	 
	  It is easy to arrange that $A_{0}$ has smooth boundary and  satisfies $B(1-\dlt') \subset A_{0} \subset B(1+\dlt')$ as well as the assumptions of Theorem \ref{thm:B} for $\xi_0={\mathbf{1}}_{A_0}$. Applying  Theorems \ref{thm:fila-patch}, \ref{thm:fila-patch-2},  we conclude 
	  $$
	\mbox{length}(\phi(t,L)) \geq c t,\quad \forall t\geq 0 
	  $$  for some $c = c(\dlt')>0$.
	  It implies $\mathrm{diam}(A(t)), \mathrm{perim}(A_{cs}(t)) \ge ct.$ 
\end{proof}
  
As an application of Theorem \ref{thm:fila-patch-2}, we can deduce large growth of the inverse of the \textit{inscription radius}: given an axi-symmetric connected open set $A\subset \bbR^{3}$ with smooth boundary, let us define $r_{ins}(A)$ to be the radius of the largest \textit{axi-symmetric} sphere $S$ contained in $\overline{A}$. Growth in time of the inverse of this radius is a way to quantify small scale creation.
\begin{corollary}[Growth of the inverse of inscription radius; {{cf. \cite[Figure 7]{pozrikidis_1986}}}] \label{cor:ins_growth}
	For any $\dlt'>0$, there exists an axi-symmetric open set $A_{0}$ with $C^{\infty}$--smooth boundary with $\left(r_{ins}(A_{0})\right)^{-1}\le 2$ such that the solution $\mathbf{1}_{A(t)}$ satisfies \begin{equation}\label{eq:ins}
		\begin{split}
			\sup_{ t \in [0,\infty) } \frac{1}{r_{ins}(A(t))} \ge (\dlt')^{-1}.
		\end{split}
	\end{equation} 
\end{corollary}

\begin{proof}
	The initial set $A_{0}$ can be obtained as follows: first fix some small $\eta>0$ and 
	cut out a small region near $z=-1$ from $B$ (the unit ball centered at the origin) so that $A_{0} \cap \{ r = 0 \} = (-1+\eta,1)$. For any $\varep>0$, one can further arrange that the assumptions of Theorem \ref{thm:B}   is satisfied for $\xi_0=\mathbf{1}_{A_{0}}$ with $\dlt=\dlt(\varep)>0$. We then apply Theorem \ref{thm:fila-patch-2} to the trajectory $\phi(t,x_{0})$ where $x_{0} = (0, -1+\eta) \in \partial A_{0}$. The inequality \eqref{eq:ins} follows by taking $\varep>0$ very small in a way depending on $\dlt'$, and considering the inscription radius at time $t = T(\varep,\eta)$ given in the case 1 of Theorem \ref{thm:fila-patch-2}. We omit the details. 
\end{proof} 

\begin{proof}[Proof of Theorem \ref{thm:fila-patch}]
Fix $\eta\in(0,1)$ and let $0<\varepsilon< \varepsilon_{0}$ where the constant $\varepsilon_{0}$ comes from Theorem \ref{thm:B}. We shall take $\varepsilon_{0}>0$ smaller (in a way depending on $\eta$) in the following. 
	We begin with a simple result. 
	\begin{lemma}\label{lem:dist-vel}
		Given a point $x_{0} = (0, z_{0})$, we set $d =|z_{0}-\tau(t)|$. Then \begin{equation*}
			\begin{split}
				\left| u^z(t, x_{0} ) - \frac{W}{d^3}\right| \le C_0\varepsilon^{\frac12}, \quad d \ge 1
			\end{split}
		\end{equation*} and \begin{equation*}
		\begin{split}
			\left| u^z(t, x_{0} ) - \frac{W}{2}(5-3d^{2})\right| \le C_0\varepsilon^{\frac12}, \quad d \le 1.
		\end{split}
	\end{equation*} 
	\end{lemma}
	
	\noindent This lemma follows immediately from \eqref{vel_explicit} and Lemma \ref{lem:FS}.  Then, we shall set up the following bootstrap hypothesis:
	
	\medskip
	
	\noindent \textbf{Assumption.} For some $\mu=\mu(\eta)>0$, $\phi^{z}(t,x_{0}) - (W - \mu)t < -1 - \frac{\eta}{4}$.
	
	\medskip
	
	\noindent By continuity of the flow, the above assumption is satisfied with $\mu=\eta$ (say) for some small time interval containing $t = 0$. We proceed under the assumption that the above holds on $[0,t^*]$ for some $t^*>0$. We need to take $\mu>0$ smaller in the following argument, but in a way depending only on $\eta$. In the following we shall restrict $t \in [0,t^*]$. From  \eqref{main_est} (with $M=1$), we have $$ |\tau(t) -Wt| \le  {C_*\sqrt{\varepsilon}(t+1)},\quad t\geq 0$$ for some absolute constant $C_*>0$. Therefore, if $\varep_0$ is taken sufficiently small relative to $\eta$ and $\mu$, we obtain from the bootstrap assumption that \begin{equation*}
		\begin{split}
			d(t) :=|\phi^{z}(t,x_{0})-\tau(t)| > 1
		\end{split}
	\end{equation*} 
holds for all $t\in[0,t^*]$.	
	Therefore, we may apply Lemma \ref{lem:dist-vel} to obtain \begin{equation*}
		\begin{split}
			\frac{d}{dt} \phi^{z}(t,x_0) & \le \frac{W}{|\phi^z - \tau(t)|^{3}} + C_0 \varepsilon^{\frac12}  \\
			& \le  \frac{W}{|\phi^z - (Wt-
 {C_*\sqrt{\varepsilon}(t+1)}			
)|^{3}			} + C_0 \varepsilon^{\frac12} ,
		\end{split}
	\end{equation*} and since \begin{equation*}
	\begin{split}
		\phi^z - Wt + 
 {C_*\sqrt{\varepsilon}(t+1)}		
		< - \mu t - 1 - \frac{\eta}{4} +  {C_*\sqrt{\varepsilon}(t+1)}		 <  - 1 - \frac{\eta}{8} < 0
	\end{split}
\end{equation*} for $\varep_0$ small relative to $\eta$ and $\mu$, we obtain \begin{equation*}
\begin{split}
		\frac{d}{dt}( \phi^{z}(t,x_0) - (W-\mu)t )& \le \frac{W}{| {1+(\eta/8)}|^3} -(W-\mu) + C_0 \varep^{\frac12} .
\end{split}
\end{equation*} Observe that the right hand side can be negative if $\mu$ and $\varepsilon_0$ are taken sufficiently small relative to $\eta$. Therefore, by integrating in time from $t = 0$ to $t^*$, we obtain \begin{equation*}
\begin{split}
	\phi^{z}(t^*,x_0) - (W-\mu)t^* < z_0 <-1 - \eta
\end{split}
\end{equation*} which gives \textbf{Assumption} with $\eta$ rather than $\eta/4$. This confirms that the bootstrap hypothesis is valid for all $t\ge0$.  
 \end{proof}

\begin{proof}[Proof of Theorem \ref{thm:fila-patch-2}]
	
Let $\varepsilon\in(0,\varepsilon_0)$ where $\varepsilon_0$ is the constant from Theorem \ref{thm:B}. We shall take $\varepsilon_0$ smaller (depending on $\eta>0$) in the following. 
	To treat case 1, 
we take $C_*>0$ from  \eqref{main_est} (with $M=1$) so that $$ |\tau(t) -Wt| \le  {C_*\sqrt{\varepsilon}(t+1)},\quad t\geq 0$$ holds. We now take a constant $ {C_2\geq 4C_*/W}$ and assume that 
\begin{equation}\label{eq:c2-aa}
		\begin{split}
			\sup_{t' \in [0,t]} \left|\phi^{z}(t',x_{0}) - \tau(t') \right|< 1 - C_{2} \varep^{\frac12}
		\end{split}
	\end{equation} holds up to some $t=t^*>0$. Indeed, it is clear that  \eqref{eq:c2-aa}  holds for some small time interval, for sufficiently small $\varep>0$ (relative to $\eta$).  Then, on $t\in[0,t^*]$, we have \begin{equation*}
		\begin{split}
			d(t) :=|\phi^{z}(t,x_{0})-\tau(t)| <1 
		\end{split}
	\end{equation*} and we may apply Lemma \ref{lem:dist-vel} to obtain \begin{equation*}
		\begin{split}
			\left| u(t, x_{0} ) - \frac{W}{2}(5-3d^2(t))\right| \le C_0\varepsilon^{\frac12}. 
		\end{split}
	\end{equation*} This gives \begin{equation*}
		\begin{split}
			\frac{d}{dt} \phi^z(t,x_{0}) & \ge \frac{W}{2}(5-3|\phi^{z}(t,x_{0})-\tau(t)|^{2} ) - C_{0}\varep^{\frac12} \\
			& \ge \frac{W}{2}(5 - 3(1+(C_{2})^2\varep - 2C_{2}\varep^{\frac12}) )- C_{0}\varep^{\frac12} \\
			& \ge  {W + \frac{WC_{2}}{2}\varep^{\frac12}
			\ge W + 2C_*\varep^{\frac12}}, 
		\end{split}
	\end{equation*} if $C_{2}$ and $\varep$   are taken  larger and smaller   respectively in a way depending on $W,  C_{*}$. Therefore, we have on $[0,t^*]$ that  \begin{equation*}
	\begin{split}
		\phi^z(t,x_{0}) > (W+2C_*\varep^{\frac12})t - 1 + \eta , 
	\end{split}
\end{equation*} and by comparing this with \begin{equation*}
		\begin{split}
			\tau(t) \le (W+C_*\varep^{\frac12})t + C_* {\sqrt{\varep}}, 
		\end{split}
	\end{equation*} it is not difficult to see (with a simple bootstrap argument) that there is $T>0$ satisfying \eqref{est_infil}.
	  The case 2  {and the last estimate \eqref{bothclaim}} can be proved in a similar fashion (bootstrap argument), and $T>0$ can be chosen so that 
it depends only on 	  $\varep,\eta$. We omit the proof. 
\end{proof}

\subsection{General solutions}

With  minor modifications, one may prove that the statements of Theorems \ref{thm:fila-patch} and \ref{thm:fila-patch-2} hold when the patch initial data $\xi_0=\mathbf{1}_{A_{0}}$ is replaced with a general non-negative function $\xi_{0}$ satisfying the assumptions of Theorem \ref{thm:B}. Of course, the resulting constants depend also on the constant $M$ satisfying $\|\xi_0\|_{L^\infty}\leq M$.
As a consequence,  we obtain an infinite time growth in the Hessian of a smooth vorticity (Corollary \ref{cor:hessian_growth}). 

\begin{proof}[Proof of Corollary \ref{cor:hessian_growth}]
 
	We may take an axi-symmetric non-negative initial datum $\xi_{0} \in C^{\infty}_{c}(\bbR^{3})$ satisfying the assumptions of Theorem \ref{thm:B} with a sufficiently small $\varep>0$. We further impose that \begin{itemize}
		\item $\xi_{0} \le 1$, \item 
$  \supp(\xi_{0}-\xi_{H}) \subset \left(B(1+\dlt')\setminus B(1-\dlt')\right)$		
		with $|\supp(\xi_{0})| \le 10$, where $|\cdot|$ is the Lebesgue measure of sets in $\mathbb{R}^3$, and
		\item $\xi_{0}(r,z) = 1$ when  $r=0$ and  $-1-(\dlt'/2) \le z \le 0$. 
	\end{itemize} Then, since the corresponding smooth solution $\xi$ is constant along particle trajectories, applying Theorems \ref{thm:fila-patch} and \ref{thm:fila-patch-2} with $z_{0} = -1-(\dlt'/2)$ and $z_{0} = 0$ respectively (in the case of general data), we have that the solution $\xi(t,\cdot)$ satisfies \begin{equation*}
	\begin{split}
	\xi(t,\cdot)=1 
	\end{split}
\end{equation*} on an interval of length $\gtrsim t$ lying on the symmetry axis $r=0$. However, since $|\supp(\xi(t,\cdot))|  = |\supp(\xi_{0})|  \le 10$, there exists some $z^*=z^*(t)$ for each $t\geq0$ such that $\xi(t,(0,z^*))=1$ but $\xi(t,(\tilde{r},z^*))=0$ for some $ { \tilde{r}(t) \gtrsim 1/\sqrt{t}}$. Then by the mean value theorem, we have that $$-\rd_{r}\xi(t,(r^*,z^*)) \gtrsim   { \sqrt{t}}$$ for some $0<r^*(t)\le \tilde{r}(t)$.\\

 Next, from $(\rd_r \omg^\tht)(r=0)=\xi(r=0)$, we have that \begin{equation*}
		\begin{split}
			-\rd_r\xi + \left( \frac{\xi(r=0)}{r} - \frac{\xi}{r} \right) = \frac{\rd_r\omg^\tht (r=0) - \rd_r\omg^\tht}{r}. 
		\end{split}
	\end{equation*} We fix some $t>0$ and evaluate both sides at $(r^*,z^*)$. The right hand side is bounded in absolute value by $\nrm{ \rd_{rr} \omg^{\tht}(t,\cdot)}_{L^{\infty}}$, whereas the left hand side is bounded from below by $-\rd_r\xi$ (we have used that $\xi(t,\cdot) \le 1$ and $\xi(t,(0,z^*))=1$), we conclude \begin{equation}\label{eq:vort-Hessian-growth}
	\begin{split}
		\nrm{ \rd_{rr} \omg^{\tht}(t,\cdot)}_{L^{\infty}}  \ge c  {\sqrt{t}},\quad t\geq0. 
	\end{split}
\end{equation}     Then, it is not difficult to see using $\omg = \omg^\tht e_\tht$ and \eqref{eq:vort-Hessian-growth} that the corresponding vorticity satisfies $\omg(t,\cdot) \in C^{\infty}_{c}(\bbR^{3})$ and  \begin{equation*}
		\begin{split}
			\nrm{ \nb^2 \omg (t,\cdot)}_{L^{\infty}(\mathbb{R}^3)} \ge c {\sqrt{t}}, \quad t\ge 0, 
		\end{split}
	\end{equation*} where $\nb^2$ is taken in the Cartesian coordinates in $\mathbb{R}^3$. 
\end{proof}

\section*{Acknowledgement}

\noindent KC has been supported by the National Research Foundation of Korea (NRF-2018R1D1A1B07043065) and by the UBSI Research Fund(1.219114.01) of UNIST. IJ has been supported by the Samsung Science and Technology Foundation under Project Number SSTF-BA2002-04 and the New Faculty Startup Fund from Seoul National University. 

\ \\ 
\bibliographystyle{abbrv}
\bibliography{bib_Instability_0203_2022_hill}

\end{document}